\documentclass[11pt,leqno]{article}
\usepackage{amsmath,amssymb,amscd,latexsym,bm}

\usepackage{xcolor}

\usepackage{changebar}

\usepackage[all]{xy}

\newcommand{\C}{{\mathbb{C}}}
\newcommand{\F}{{\mathbb{F}}}
\newcommand{\oF}{\overline{\F}}
\newcommand{\Ge}{\mathbb{G}}
\newcommand{\Pa}{{\mathbb{P}}}
\newcommand{\Q}{{\mathbb{Q}}}
\newcommand{\oQ}{\overline{\Q}}

\newcommand{\Z}{{\mathbb{Z}}}
\newcommand{\oZ}{\overline{\Z}}

\newcommand{\tQ}{\tilde{Q}}

\newcommand{\Alb}{\mathrm{Alb}}
\newcommand{\ch}{\mathrm{ch}\,}

\newcommand{\et}{\mathrm{\acute{e}t}}

\newcommand{\Free}{\mathrm{Free}}
\newcommand{\id}{\mathrm{id}}

\renewcommand{\mod}{\mathrm{mod}\,}
\newcommand{\Mor}{\mathrm{Mor}}

\newcommand{\red}{\mathrm{red}}
\newcommand{\rk}{\mathrm{rk}\,}
\newcommand{\Spec}{\mathrm{Spec}\,}
\newcommand{\spec}{\mathrm{Spec}\,}
\newcommand{\ssp}{\mathrm{sp}}

\newcommand{\Aut}{\mathrm{Aut}}

\newcommand{\Ext}{\mathrm{Ext}}
\newcommand{\Fr}{\mathrm{Fr}}

\newcommand{\Gal}{\mathrm{Gal}}
\newcommand{\GL}{\mathrm{GL}\,}
\newcommand{\bGL}{\mathbf{GL}}
\newcommand{\Hom}{\mathrm{Hom}}
\newcommand{\uHom}{\underline{\Hom}}
\newcommand{\Imm}{\mathrm{Im}\,}
\newcommand{\Iso}{\mathrm{Iso}\,}
\newcommand{\Ker}{\mathrm{Ker}\,}

\newcommand{\Ob}{\mathrm{Ob}\,}
\newcommand{\Pic}{\mathrm{Pic}}
\newcommand{\PGL}{\mathrm{PGL}}
\newcommand{\rank}{\mathrm{rank}}
\newcommand{\Rep}{\mathbf{Rep}\,}
\newcommand{\ttop}{\mathrm{top}}
\newcommand{\uVec}{\mathbf{Vec}}
\newcommand{\nVec}{\mathrm{Vec}}
\newcommand{\Ch}{{\mathcal C}}
\newcommand{\Bh}{{\mathcal B}}

\newcommand{\Eh}{{\mathcal E}}
\newcommand{\tEh}{\tilde{\Eh}}
\newcommand{\Gh}{{\mathcal G}}

\newcommand{\Kh}{\mathcal{K}}
\newcommand{\oKh}{\overline{\Kh}}

\newcommand{\Mh}{{\mathcal M}}

\newcommand{\Oh}{{\mathcal O}}

\newcommand{\Th}{{\mathcal T}}
\newcommand{\Wh}{\mathcal{W}}
\newcommand{\Yh}{\mathcal{Y}}
\newcommand{\Xh}{{\cal X}}
\newcommand{\Zh}{\mathcal{Z}}

\newcommand{\emm}{{\mathfrak{m}}}
\newcommand{\en}{\mathfrak{n}}
\newcommand{\eo}{\mathfrak{o}}
\newcommand{\ep}{\mathfrak{p}}

\newcommand{\eU}{\mathfrak{U}}
\newcommand{\eX}{{\mathfrak X}}
\newcommand{\oeX}{\overline{\eX}}
\newcommand{\teX}{\tilde{\eX}}

\newcommand{\oK}{\overline{K}}

\newcommand{\tYh}{\tilde{\Yh}}

\newcommand{\tgamma}{\tilde{\gamma}}
\newcommand{\blambda}{\bm{\lambda}}
\newcommand{\hsi}{{}^{\sigma}\!}
\newcommand{\tpi}{\tilde{\pi}}
\newcommand{\opi}{\overline{\pi}}

\newcommand{\talpha}{\tilde{\alpha}}

\newcommand{\tU}{\tilde{U}}

\newcommand{\tvarphi}{\tilde{\varphi}}
\newcommand{\oeta}{\overline{\eta}}
\newcommand{\tV}{\tilde{V}}
\newcommand{\tx}{\tilde{x}}
\newcommand{\ty}{\tilde{y}}

\newcommand{\tZ}{\tilde{Z}}

\newcommand{\oQu}{\overline{Q}}

\newcommand{\ohne}{\setminus}
\newcommand{\iso}{\stackrel{\sim}{\rightarrow}}
\newtheorem{theorem}{Theorem}[section]
\newtheorem{lemma}[theorem]{Lemma}
\newtheorem{prop}[theorem]{Proposition}
\newtheorem{defn}[theorem]{Definition}
\newtheorem{cor}[theorem]{Corollary}

\newtheorem{notation}[theorem]{Notation}
\newtheorem{constr}[theorem]{Construction}

\newenvironment{claim}{\noindent {\bf Claim}}{}
\newenvironment{rem}{\noindent {\bf Remark}}{}
\newenvironment{rems}{\noindent {\bf Remarks}}{}

\newenvironment{theoremon}{\noindent {\bf Theorem}\it}{}
\newenvironment{theorembhatt}{\noindent {\bf Theorem \ref{bhatt} }\it}{}
\newenvironment{proofof}{\noindent {\bf Proof of}}{\mbox{}\hspace*{\fill}$\Box$}
\newenvironment{proof}{\noindent {\bf Proof}}{\mbox{}\hspace*{\fill}$\Box$}

\newcommand{\pr}{\mathrm{pr}}
\newcommand{\tFr}{\tilde{\Fr}}
\newcommand{\tX}{\tilde{X}}
\newcommand{\Fh}{\mathcal{F}}
\newcommand{\Quot}{\mathrm{Quot}}
\newcommand{\uEnd}{\underline{\mathrm{End}}}
\newcommand{\verk}{\mbox{\scriptsize $\,\circ\,$}}

\begin{document}
\title{Parallel transport for vector bundles on  $p$-adic varieties}
\author{Christopher Deninger \and Annette Werner}
\date{}
\maketitle

\centerline{\bf Abstract:} We develop a theory of \'etale parallel transport for vector bundles with numerically flat reduction on a $p$-adic variety. 
This construction is compatible with natural operations on vector bundles, Galois equivariant and functorial with respect to morphisms of varieties. 
In particular, it provides a continuous $p$-adic representation of the \'etale fundamental group for every vector bundle with numerically flat reduction.
The results in the present paper generalize previous work by the authors on curves.  They can be seen as a $p$-adic analog of higher-dimensional 
generalizations of the classical Narasimhan-Seshadri correspondence on complex varieties. Moreover, they provide new insights into Faltings' $p$-adic Simpson correspondence between
small Higgs bundles and small generalized representations  by establishing a class of vector bundles with vanishing Higgs field giving rise to actual (not only generalized) representations.

\small 
~\\[0.3cm]

\centerline{{\bf 2010 MSC: 14J20, 11G25} } 

\section{Introduction} \label{intro}

In the present paper we develop a theory of \'etale parallel transport for suitable vector bundles on a  $p$-adic variety $X$. The vector bundles we consider are those with numerically flat (equivalently, Nori-semistable) reduction on a  model of $X$ over the ring of integers. Referring to later sections for the precise definitions let us begin by stating our main result, which combines Proposition \ref{cd-t35} and Theorem \ref{cd-t36}.

\begin{theoremon} Let $X$ be a smooth, complete variety over $\oQ_p$ and let $E$ be a vector bundle on $X \otimes \C_p$ with numerically flat reduction. Then for any \'etale path $\gamma$ from a point $x \in X (\C_p)$ to a point $x' \in X (\C_p)$ our construction below gives a canonical isomorphism
\[
\rho_E (\gamma) : E_x \iso E_{x'} \; .
\]
We have $\rho_E (\gamma_1 \cdot \gamma_2) = \rho_E (\gamma_1) \verk \rho_E (\gamma_2)$ if the paths $\gamma_1$ and $\gamma_2$ can be composed. The parallel transport $\rho_E (\gamma)$ depends continuously on $\gamma$ and it is compatible with direct sums, tensor products, duals, internal homs and exterior powers. Moreover it is Galois equivariant and functorial with respect to morphisms of varieties. In particular we obtain a continuous representation
\begin{equation}
%\label{eq:i1}
\rho_{E,x} : \pi_1 (X,x) \longrightarrow \GL (E_x)
\end{equation}
for every $x \in X (\C_p)$. 
\end{theoremon}

In fact, all vector bundles considered in this result are numerically flat on $X \otimes \C_p$ by Theorem \ref{cd-t38}. If the smooth variety $X \otimes \C_p$ is projective this is equivalent to the fact that they have vanishing Chern classes and are slope
semistable with respect to one (equivalently: all) polarizations. The theorem generalizes previous results for curves \cite{dewe1}. Note that our parallel transport can be seen as a $p$-adic analog of generalizations of the classical Narasimhan--Seshadri correspondence to higher dimensional complex varieties as in  \cite{D}, \cite{MR} and \cite{UY}. In the complex case, this theory is the special case of the Simpson correspondence \cite{S} in the case of vanishing Higgs field. There exists also a $p$-adic version of Simpson's correspondence developed by Faltings \cite{Fa2} and \cite{fa3}. A detailed and systematic treatment is provided by \cite{agt}. If $X$ is a curve, this correspondence relates Higgs bundles and generalized representations. It is shown in \cite{xu} that for curves our parallel transport is an inverse to Faltings' functor for a category of of suitable vector bundles with trivial Higgs field. It is an open problem, however, to determine the category of Higgs bundles corresponding to actual $p$-adic representations in the $p$-adic Simpson correspondence.

In higher dimensions, Faltings' $p$-adic Simpson correspondence gives an equivalence between small Higgs bundles and small generalized representations. Our main new insight into the zero Higgs field case of \cite{Fa2} is the construction of actual (not merely generalized)  
representations associated to a natural category of vector bundles $E$. Apart from this, our result provides a functor of parallel transport which  is a stronger structure than a representation of the fundamental group. Note that Liu and Zhu \cite{liuzhu} establish a Riemann-Hilbert functor on rigid analytic varieties which yields part of a $p$-adic Simpson correspondence, namely a tensor functor from the category of \'etale $\Q_p$-local systems to the category of nilpotent Higgs bundles. We also draw the reader's attention to the paper \cite{ogus} establishing a characteristic $p$ analog of Simpson's nonabelian Hodge theory and to the related construction of Higgs-deRham flows in \cite{lsz}.

In the present paper, the parallel transport morphisms are defined by taking a $p$-adic limit over morphisms modulo all $p^n$. This relies on  infinitely many reduction conditions on the vector bundle. An important step of the proof of the theorem is therefore to show that the numerical flatness condition we impose on the reduction of a vector bundle implies these 
infinitely many reduction conditions on suitable covers of the original model. This requires in particular the construction of covers on which certain cohomology obstructions vanish. Here we need a strengthening of Bhatt's result \cite{bh1} which is due to Bhatt and Snowdon \cite{bh2}.

It is an interesting open problem to find a semistability condition for a vector bundle on the smooth projective variety $X$ itself, i.e. not involving models, which guarantees the existence of \'etale parallel transport. 

Let us now describe the contents of the present paper in more detail. 
Section \ref{sec:cd1} deals with numerically flat vector bundles on a projective scheme $X$  over a field $k$. Later on, we want to apply these results to special fibers of integral models of our $p$-adic variety. By definition, a vector bundle $E$ on $X$ is called numerically flat, if $E$ and its dual are nef. This is equivalent to the fact that for all morphisms
 $f : C \to X$ from smooth projective curves $C$  to $X$ over $k$, the bundle $f^* E$ is semistable of degree zero. Generalizing a boundedness result by Langer, we show that for $k = \oF_p$ numerically flat bundles have vanishing Chern classes and become trivial after pullback to some $Y \rightarrow X$  which is the composition of a finite \'etale morphism with a Frobenius (see Theorems \ref{cd-t2} and \ref{cd-t3}). As a byproduct of the work in Section \ref{sec:cd1} we show that the $S$-fundamental group and Nori's fundamental group are isomorphic group schemes for every projective connected reduced scheme over a finite field or its algebraic closure. Formerly this was known in the smooth case.

In Section \ref{sec:cd2} we provide the necessary background on models of $p$-adic varieties over the ring of integers. We show how to pass to projective models with reduced special fibers and smooth generic fibers in Lemma \ref{cd-t13}. Here we have to deal with a number of technical difficulties, e.g. we need equivariant resolution of singularities. 

In Section \ref{sec:cd3} we provide the formal framework for constructing parallel transport modulo $p^n$. In particular, we prove a Seifert-van Kampen theorem for the groupoid of \'etale paths.  
 This kind of gluing result is necessary since later on we work with covers of models which are  only locally finite \'etale in the generic fiber. 
 Our basic construction is given in \ref{cd-t26}: Let $\Eh$ be a vector bundle on a model $\eX$ of a $p$-adic variety and assume  that there exists a suitable morphism $\pi: \Yh \rightarrow \eX$ which is finite \'etale surjective over some open dense subset $U$ of the generic fiber of $\eX$ and which has the property that $\pi^\ast \Eh$ is trivial modulo $p^n$. Then $\pi^\ast \Eh$ can be endowed with a trivial parallel transport modulo $p^n$, and this descends naturally to parallel transport along \'etale paths in $U$.

The following three sections show how to apply this construction to vector bundles with numerically flat reduction. In Section \ref{section5} we prove that every vector bundle $\Eh$ on a model $\eX$ of a smooth and projective $\oQ_p$-variety $X$ which has numerically trivial special fiber gives rise to the following data: an open covering $\{U_i\}$ of $X$ together with complete morphisms $\pi_i: \Yh_i \rightarrow \eX$ which are finite \'etale and surjective over $U_i$ such that $\pi_i^\ast \Eh$ has trivial special fiber (see Theorem \ref{cd-t22}). In order to prove this theorem, we use the fact that a numerically flat bundle on the special fiber of $\eX$ can be trivialized by the composition of a finite \'etale covering of the special fiber followed by a Frobenius. By a careful geometric argument involving a naive Frobenius on projective space we succeed to lift a suitable  covering of the Frobenius to characteristic zero, which gives the desired trivializing covers.  

The goal of Section \ref{sec:cd6} is the following theorem, which is a strengthening of \cite{bh1}, Theorem 1.2 over  a complete discrete valuation ring  $\eo_K$ of mixed characteristic and perfect residue field $k$. 

\begin{theorembhatt} Let $\eX$ be a reduced and proper scheme over $\eo_K$, and let $F \subset \eX$ a finite set of closed points which is contained in a connected, quasi-projective, smooth open subset of the generic fiber $\eX_K$. 
There is an $\eo_K$-scheme $\Zh$ and a proper $\eo_K$-morphism $f:\Zh \rightarrow \eX$ such that

i) For all $i \ge 1$, the pullback $f^\ast: H^i(\eX, {\cal O}_\eX) \rightarrow H^i (\Zh, {\cal O}_\Zh)$ is divisible by $p$ as a group homomorphism.

ii) The map  $f$ is finite \'etale surjective around $F$. 

\end{theorembhatt}

In fact, the proof follows the proof of \cite{bh1}, Theorem 1.2 paying close attention to the alterations involved in order to establish the new condition ii). 
For this purpose the paper \cite{bh2} by Bhatt and Snowdon is a crucial tool. 

In Section \ref{sec:cd5}, we use the previous theorem in order to prove the following result (which combines Theorem \ref{cd-t22} and Theorem \ref{cd-t23}).

\begin{theoremon} Let $X$ be a smooth complete variety over $\oQ_p$ and $\eX$ a model of $X$ over $\oZ_p$. Let $\Eh$ be a vector bundle on $\eX$ with numerically flat special fiber. For every $n \geq 1$ there exists an open covering $\{U_i\}$ of $X$ together with complete morphisms $\pi_i: \Yh_i \rightarrow \eX$ which are finite \'etale and surjective over $U_i$ such that $\pi_i^\ast \Eh$ is trivial modulo $p^n$.
\end{theoremon}

In order to prove this theorem, we start with the covering which trivializes the special fiber of $\Eh$ and lift it to coverings trivializing $\Eh$ modulo $p^n$ by
 killing obstructions in $H^1 (\eX \otimes \oZ_p / p^n \oZ_p , \Oh)$ and $H^2 (\eX \otimes \oZ_p / p^n \oZ_p, \Oh)$ with the help of Theorem \ref{bhatt}.  
 
Now our basic construction \ref{cd-t26} applies. We can glue with the help of our Seifert-van Kampen result and pass to the  $p$-adic limit in order to define in Section \ref{sec:cd31} isomorphisms of $p$-adic parallel transport along \'etale paths for all vector bundles $\Eh$ with numerically flat reduction on a model $\eX$, see Theorem \ref{cd-t30n} and Proposition \ref{prop-rho}.  In Section \ref{sec:cd4} we vary the models and deduce our main theorem above.  

Section \ref{sec:cd8} deals with a straightforward generalization of our theory to a bigger class of vector bundles which in particular contains all line bundles with vanishing rational first Chern class.

{\it Acknowledgements.} We are very grateful to Bhargav Bhatt, Johan de Jong, Adrian Langer  and Stefan Wewers for helpful discussions and email exchanges. Many thanks go to Bhargav Bhatt for kindly refining the main theorem of \cite{bh1} in \cite{bh2}.
\section{Numerically flat bundles} \label{sec:cd1}
Throughout this paper, a variety over a field $k$ is a geometrically integral separated scheme of finite type over $k$. 
Recall that a scheme is called integral if it is irreducible and reduced.

We want to study the following class of vector bundles which was introduced in \cite{DM} (2.34).

\begin{defn}
\label{cd-t1}
Let $X$ be a connected and complete scheme over a perfect field $k$. A vector bundle $E$ over $X$ is called Nori-semistable if for all $k$-morphisms $f : C \to X$ from smooth connected projective curves $C$ over $k$ into $X$ the pullback $f^* E$ is semistable of degree zero.
\end{defn}

A related but different notion is due to Nori \cite{N} who considers only pullbacks to embedded smooth projective curves. In \cite{DM} the bundles in Definition 1 are simply called semistable. Since this can lead to misunderstandings, the name Nori-semistable has been adopted by several authors. 

By definition $E$ is Nori-semistable on $X$ if and only if $E \, |_{X^{\red}}$ is Nori-semistable on $X^{\red}$. Moreover for a $k$-morphism $g : Y \to X$ of complete connected $k$-schemes and a Nori-semistable bundle $E$ on $X$ the pullback $g^* E$ is a Nori-semistable bundle on $Y$.

The bundles in question have an internal characterization. Recall that a vector bundle $E$ on a connected complete $k$-scheme $X$ is called nef if the line bundle $\Oh_{\Pa (E)} (1)$ on $\Pa (E)$ is nef. 

The vector bundle $E$ is called numerically flat if $E$ and its dual bundle $E^*$ are nef. It follows from the definition that if $g : Y \to X$ is a surjective $k$-morphism of complete connected $k$-schemes and if $g^* E$ is numerically flat then the vector bundle $E$ on $X$ is numerically flat as well.

It is known that $E$ is a Nori-semistable bundle if and only if $E$ is numerically flat, see \cite{L1} 1.2.

A line bundle $L$ on a projective integral scheme $X$ over $k$ is numerically flat if and only if it is in $\Pic^{\tau} X$ i.e. if a tensor power $L^N$ is in $\Pic^0 X$, see \cite{L2} 1.3.

Consider a complete algebraic scheme $X$ over a field $k$ and let $A_* (X)$ and $A^* (X) = A^* (X \xrightarrow{\id} X)$ be the homology and the (operational) cohomology theory, respectively,  of algebraic cycles modulo rational equivalence, \cite{F} Chapter 17. There is a cap product pairing
\[
\cap : A^i (X) \times A_j (X) \longrightarrow A_{j-i} (X)
\]
and the well-known degree map on zero cycles
\[
\int_X : A_0 (X) \longrightarrow \Z.
\]
The singular Grothendieck-Riemann-Roch theorem of Baum--Fulton--MacPherson implies that for all vector bundles $E$ on $X$ we have
\begin{equation}
\label{cd:1}
\chi (X, E) = \int_X ch (E) \cap Td (X)
\end{equation}
for a certain class $Td (X)$ in $A_* (X)_{\Q}$, see \cite{F} Corollary 18.3.1. Here $\ch (E) \in A^* (X)_{\Q}$ is the (operational) Chern character of \cite{F} Example 3.2.3. It satisfies the formula
\[
\ch (E \otimes F) = \ch (E) \, \ch (F)
\]
for vector bundles $E$ and $F$. Thus if $E$ has rank $r$ and all Chern classes of $E$ in $A^* (X)_{\Q}$ vanish, we have
\begin{equation}
\label{cd:2}
\chi (X, E \otimes F) = r \chi (X,F) \; .
\end{equation}

%It is known that numerically flat vector bundles $E$ on smooth projective varieties have vanishing Chern classes $c_i (E)$ in the (torsion-free) quotient $N^i (X)$ of $A^i (X) \cong A_{d-i} (X)$ by the relation of numerical equivalence. Over fields of characteristic zero the proof uses transcendental methods \cite{}. Over fields of positive characteristic $p$, Langer argues as follows, \cite{} proof of Theorem 2.2: Firstly the set of numerically flat vector bundles of a given rank is bounded by \cite{L2} Theorem 1.1. Next, he notes that a bounded family has only finitely many Chern classes in $N^* (X)$, c.f. \cite{} Lemma 1.1. Now the equation $c_i (\Fr^{\nu *} E) = p^{\nu i} c_i (E)$ implies that $c_i (E) = 0$ for all $i$ since the Frobenius pullbacks $\Fr^{\nu *} E$ for $\nu \ge 1$ are again numerically flat and $N^* (X)$ has no $\Z$-torsion. 

In the case where the base field $k$ is finite or where $k = \oF_p$ there is the following characterization of numerically flat vector bundles:

\begin{theorem}
\label{cd-t2}
Let $E$ be a vector bundle on a projective connected scheme $X$ over $k = \F_q$ or $k = \oF_p$. Then the following properties are equivalent:\\
i) $E$ is numerically flat (equivalently Nori-semistable).\\
ii) There are a projective connected scheme $Y$ over $k$, a finite \'etale morphism $\pi : Y \to X$, and an integer $\nu \ge 0$, such that for the composition
\[
\varphi : Y \xrightarrow{F^{\nu}} Y \xrightarrow {\pi} X
\]
the pullback $\varphi^* E$ is a trivial bundle. Here for $F$ we may either take the absolute Frobenius morphism $F = \Fr_p$ or a $k$-linear Frobenius morphism $F = \Fr_q = \Fr^r_p$ where $q = p^r$ (for $k = \F_q$) resp. $F = \Fr_q \otimes_{\F_q} \id_k$ (for $k = \oF_p$) corresponding to an isomorphism $Y = Y_0 \otimes_{\F_q} k$ with $Y_0$ a proper scheme over $\F_q$. 
\end{theorem}

The theorem will be proved together with the following result

\begin{theorem}
\label{cd-t3}
Let $E$ be a numerically flat (equivalently Nori-semistable) bundle of rank $r$ on a complete connected scheme $X$ over $k = \F_q$ or $k = \oF_p$. Then $c_i (E) = 0$ in $A^* (X)_{\Q}$ for all $i \ge 1$ and $\chi (X, E \otimes F) = r \chi (X, F)$ for all vector bundles $F$ on $X$.
\end{theorem}

\begin{proof}
Since $k$ is perfect the projection $\pr_Y$ in the following diagram defining the relative Frobenius morphism $\tFr_q$ is an isomorphism
\[
\xymatrix{
Y \ar@/^1pc/[drr]^{\Fr_q} \ar[dr]^{\tFr_q} \ar@/_/[ddr] \\
 & Y \otimes_{k, ()^q} k \ar[r]^-{\overset{\pr_Y}{\sim}} \ar[d] & Y  \ar[d] \\
 & \spec k \ar[r]^{\overset{\Fr_q}{\sim}} & \spec k
}
\]
Any isomorphism $Y = Y_0 \otimes_{\F_q} k$ gives rise to a commutative square
\[
\xymatrix{
Y \ar[r]^-{\tFr_q} \ar@{=}[d] & Y \otimes_{k , ()^q} k \ar@{=}[d] \\
Y_0 \otimes_{\F_q} k \ar[r]^{\Fr_q \otimes \id_k} & Y_0 \otimes_{\F_q} k
}
\]
of $k$-morphisms. It follows that for a vector bundle on $Y$ it is equivalent to be trivialized by either of the pullbacks via $\Fr_q , \tFr_q$ or $\Fr_q \otimes \id_k$. Now assume that ii) holds. Since $\tFr_q$ and $\pi$ are surjective, the $k$-linear morphism $\tvarphi = \pi \verk \tFr^{\nu}_q$ is surjective as well. Since numerical flatness can be tested after pullback by a surjective $k$-linear morphism, it follows that ii) implies i). For the reserve implication we may assume that $k$ is finite since $E$ and $X$ can be descended to a numerically flat vector bundle on a projective connected scheme over $\F_q$ for some $q$. We first treat the case where the scheme $X$ is a projective {\it normal} variety over the finite field $k$. Using Langer's boundedness result \cite{L2} Theorem 1.1 it follows that the set of isomorphism classes of numerically flat vector bundles on $X$ of a given rank is finite. Hence there are integers $\mu > \nu \ge 0$ such that $\Fr^{\nu *}_q E \cong \Fr^{\mu *}_q E$. Setting $G = \Fr^{\nu *}_q E$ and $n = \mu - \nu$ it follows that $\Fr^{n *}_q G \cong G$. The proof of \cite{LS} Satz 1.4 extends without change to an arbitrary $\F_q$-scheme (note that finiteness is not proved there but true). As Adrian Langer pointed out to us, this result is also contained in the earlier paper \cite{ka}, see the proof of Proposition 4.1. Hence there is a finite \'etale morphism $\pi : Y \to X$ such that $\pi^* G = \pi^* \Fr^{\nu *}_q E = \Fr^{\nu *}_q \pi^* E$ is a trivial bundle. Since $X$ is connected we may assume that $Y$ is connected as well (by passing to a connected component). Moreover, since $X$ is projective and $\pi$ is finite, $Y$ is projective as well. Thus assertion ii) is proved for normal projective varieties $X$. Next we note that for any vector bundle $E$ on $X$ such that $\varphi^* E$ is trivial the Chern classes $c_i (E) \in A^* (X)_{\Q}$ vanish for $i \ge 1$. This follows from the relation $\pi_* \pi^* = d$ on $A^* (X)$ where $d$ is the degree of the finite \'etale map $\pi$ and from the formula $c_i (\Fr^*_p E) = p^i c_i (E)$. Hence Theorem \ref{cd-t3} holds if $X$ is a normal projective variety over $k = \F_q$ or $\oF_p$. To deduce Theorem \ref{cd-t3} in general we need the following notions from \cite{F} Definition 18.3, see also \cite{Ki}. An envelope of a scheme $X$ is a proper morphism $f : X' \to X$ such that for every subvariety $V$ of $X$ there is a subvariety $V'$ of $X'$ such that $f$ maps $V'$ birationally onto $V$. According to \cite{F} Example 17.3.2 the pullback morphism
\[
f^* : A^* (X) \longrightarrow A^* (X')
\]
along an envelope $f$ is injective. If $X$ and $X'$ are $K$-schemes for a field $K$ and $f$ is a morphism over $\spec K$ the envelope is called a Chow envelope if $X'$ is quasiprojective over $\spec K$. Any $K$-scheme $X$ has a Chow envelope $f : X' \to X$ by \cite{F} Lemma 18.3 (3) and we may assume that $X'$ is reduced. We claim that if $X$ is a complete $K$-scheme there exists an envelope $g : X'' \to X$ where $X''$ is the disjoint union of projective normal varieties over $\spec K$. Since the composition of envelopes is an envelope we may assume by passing to a Chow envelope that $X$ is a reduced projective $K$-scheme. The regular locus $U \subset X$ is dense in $X$. By induction on the dimension of $X$ we can assume that there is an envelope $X''_1 \to X \setminus U$ of the desired type over the singular locus of $X$. Let $X''_2$ be the disjoint union of the normalizations of the irreducible components of $X$. Then
\[
\pi : X'' := X''_1 \amalg X''_2 \longrightarrow X
\]
is the desired envelope. Hence the pullback morphism
\[
\pi^* : A^* (X) \longrightarrow A^* (X'')
\]
is injective. Consider $X$ as in Theorem \ref{cd-t3} and let $E$ be a numerically flat vector bundle on $X$. Then $\pi^* E$ is numerically flat on each connected component $X''_{(\alpha)}$ of $X''$. Since $X''_{(\alpha)}$ is a projective normal variety over $k$ it follows from what we have seen that $\pi^* c_i (E) = c_i (\pi^* E)$ is zero in $A^* (X'')_{\Q} = \bigoplus_{\alpha} A^* (X''_{(\alpha)})_{\Q}$ for $i \ge 1$. Since $\pi^*$ is injective the Chern classes $c_i (E)$ in $A^* (X)_{\Q}$ vanish as well and formula \eqref{cd:2} holds. Thus Theorem \ref{cd-t3} is proved. We will now prove the implication i) $\Rightarrow$ ii) in Theorem \ref{cd-t2} for arbitrary projective connected $k$-schemes $X$. As mentioned above it suffices to consider the case $k = \F_q$. It is enough to show that there are only finitely many isomorphism classes of numerically flat vector bundles on $X$. Then we can use the Langer--Stuhler argument as before. Note that finiteness is implied by \cite{lan04}, Theorem 4.4. Since the reduction step to the smooth case is left to the reader in loc. cit., let us give an argument here. We first assume that $X$ is reduced. Let $\tX$ be the disjoint union of the normalizations of the irreducible components of $X$ and let $\pi : \tX \to X$ be the corresponding finite surjective morphism. For a numerically flat vector bundle $E$ on $X$ the restrictions of $\pi^* E$ to the connected components of $\tX$ which are normal projective varieties over $k = \F_q$ are numerically flat. By the above, up to isomorphism there are only finitely many possibilities for them and hence for $\pi^* E$. The adjunction map $E \to \pi_* \pi^* E$ is injective and we obtain an exact sequence of coherent sheaves on $X$
\[
0 \longrightarrow E \longrightarrow \Fh \longrightarrow \Gh \longrightarrow 0
\]
where $\Fh = \pi_* \pi^* E$. Since there are only finitely many possibilities for $\Fh$ we may assume that $\Fh$ is any fixed coherent sheaf on $X$ and we have to show that up to isomorphism it has at most finitely many quotients $\Fh \twoheadrightarrow \Gh$ whose kernel is $E$. Choose a polarization $\Oh (1)$ of the projective variety $X$. By Theorem \ref{cd-t3} we have $\chi (X, E(n)) = r \chi (X , \Oh (n))$ where $r$ is the rank of $E$. Since we fixed $\Fh$, the Hilbert polynomial of $\Gh$ is fixed as well
\[
\chi (X, \Gh (n)) = \chi (X , \Fh (n)) - r \chi (X , \Oh (n)) =: P (n) \; .
\]
The isomorphism classes of the relevant quotients $\Gh$ of $\Fh$ therefore correspond to the $\F_q$-valued points of the Quot scheme $\Quot^P_{\Fh / X / \F_q}$ which exists and is of finite type over $\F_q$ by a result of Grothendieck. Hence there are only finitely many $\Gh$'s and hence $E$'s. Finally let $X$ be an arbitrary connected projective scheme over $k = \F_q$. We have just seen that up to isomorphism there are only finitely many numerically flat vector bundles on $X_{\red}$. Let $J$ be the nilideal of $\Oh_X$ and let $X_n$ be the subscheme of $X$ corresponding to $J^n$. For large enough $n$ we have $X_n = X$ and $X_1 = X_{\red}$. We will show by induction on $n$ that there are only finitely many numerically flat vector bundles $E$ on $X_n$. The case $n = 1$ being settled, assume that the assertion has been shown on $X_{n-1}$ for some $n \ge 2$. Consider the exact sequence of $\Oh_{X_n}$-module sheaves
\[
0 \longrightarrow J^{n-1} / J^n \longrightarrow \Oh_{X_n} \longrightarrow \Oh_{X_{n-1}} \longrightarrow 0 \; .
\]
Tensoring with $E$ and noting that $J^{n-1} / J^n$ is an $\Oh_{X_1} = \Oh_X / J$-module we get the exact sequence
\[
0 \longrightarrow E \, |_{X_1} \otimes J^{n-1} / J^n \longrightarrow E \longrightarrow E \, |_{X_{n-1}} \longrightarrow 0 \; .
\]
By the induction assumption it is therefore enough to show that there are only finitely many extension classes of $E \, |_{X_{n-1}}$ by $E \, |_{X_1} \otimes J^{n-1} / J^n$ on $X_n$. We have
\begin{align*}
 & \Ext^1_{X_n} (E \, |_{X_{n-1}} , E \, |_{X_1} \otimes J^{n-1} / J^n) \\
= & \Ext^1_{X_1} (E \, |_{X_1} , E \, |_{X_1} \otimes J^{n-1} / J^n) \\
= & H^1 (X_1 , \uEnd (E \, |_{X_1}) \otimes J^{n-1} / J^n) \; .
\end{align*}
Since $X_1$ is projective over $\F_q$ and the coefficients are coherent, this cohomology group is finite.
Note that the 
\end{proof}

We note the following fact which was shown in the preceding proof

\begin{theorem}
\label{cd-t4}
There are only finitely many isomorphism classes of numerically flat (or Nori-semistable) vector bundles on a projective connected scheme over a finite field.
\end{theorem}

The category of numerically flat vector bundles on a complete, reduced connected $k$-scheme $X$ with a point $x \in X (k)$ is a neutral Tannakian category with the faithful fibre functor $E \mapsto E_x$, see \cite{L1} Proposition 5.5. Langer calls its Tannakian dual $\pi^S_1 (X,x)$ the $S$-fundamental group of $X$. It is defined for any perfect field $k$. There is also Nori's fundamental group $\pi^N_1 (X,x)$ which is the Tannakian dual of the category of essentially finite vector bundles on $X$, see \cite{N}. They can be characterized as those bundles which can be trivialised by a torsor on $X$ under a finite group scheme over $k$. The essentially finite bundles form a full subcategory of the numerically flat bundles and as pointed out in \cite{L1} Lemma 6.2 one obtains a faithfully flat homomorphism
\[
\pi^S_1 (X,x) \longrightarrow \pi^N_1 (X,x) \; ,
\]
In general it is not an isomorphism.

\begin{cor}
\label{cd-t5}
For a projective connected reduced scheme $X$ over $k = \F_q$ or $k = \oF_p$ and a point $x \in X (k)$, the natural morphism
\[
\pi^S_1 (X,x) \longrightarrow \pi^N_1 (X,x)
\]
is an isomorphism of group schemes over $k$.
\end{cor}

\begin{rem}
For smooth projective varieties $X$ over $\oF_p$ the result was previously shown by Mehta, unpublished, see  \cite{L1} end of introduction.
\end{rem}

\begin{proof}
We have to show that every numerically flat vector bundle $E$ on $X$ is essentially finite. We use an argument from \cite{Su}. Consider the Tannakian subcategory $\langle E \rangle^{\otimes}$ generated by $E$ in the category of numerically flat vector bundles on $X$. The Tannakian dual of $\langle E \rangle^{\otimes}$ is the {\it monodromy group scheme} $G_E$ of $E$. Let $r$ be the rank of $E$. The $\GL_n$-torsor associated to $E$ allows a reduction of structure group to $G_E$. Thus we get a $G_E$-torsor $\pi : P \to X$ such that $\pi^* E$ is a trivial bundle. By Theorem \ref{cd-t4} there are only finitely many isomorphism classes of numerically flat vector bundles on $X$. Hence there are $\mu > \nu \ge 0$ with $F^{\nu *} E \cong F^{\mu *} E$. The same argument as in \cite{Su} now implies that $G_E$ is a finite group scheme over $k$.
\end{proof}
\section{Models} \label{sec:cd2}

For the theory of the \'etale parallel transport we require certain reduction properties for vector bundles on $p$-adic varieties. In order to formulate them we now discuss some fact about models of $p$-adic varieties and their covers. Let $R$ be a valuation ring of characteristic zero which is the filtered union of discrete valuation rings dominated by $R$. By $Q$ we denote its quotient field, and by $k$ the residue field. Then $\spec R$ has two points and hence the generic point $\spec Q$ is open in $\spec R$. 

\begin{defn} \label{cd-t6}
A proper pair $(\eX , U)$ over $R$ consists of the following data:\\
i) A model $\eX$ over $R$ i.e. an $R$-scheme which is flat, proper and of finite presentation.\\
ii) A non-empty open subscheme $U$ of $\eX_Q = \eX \otimes_R Q$ which is a smooth quasi-projective variety over $Q$.

A morphism between proper pairs $(\eX,U)$ and $(\eX', U')$ over $R$ is an $R$-morphism $\eX \rightarrow \eX'$ mapping $U$ to $U'$.

\end{defn}

Given a proper scheme $X$ over $Q$ we say that $\eX$ is a model of $X$ over $R$ if $\eX$ is a model over $R$ and there is an isomorphism $X \cong \eX \otimes_R Q$ over $Q$.

\begin{lemma} \label{cd-t7}
i) Let $\eX$ be a model over $R$ whose generic fibre $\eX_Q$ is geometrically connected. Then all geometric fibres of $\eX \to \spec R$ are geometrically connected.\\
ii) Let $\eX$ be a model over $R$ with integral generic fibre $\eX_Q$. Then $\eX$ is integral as well and hence $\eX_Q$ is dense in $\eX$.
\end{lemma}
\begin{proof}
i) This is a special case of \cite{EGAIV} Proposition (15.5.9), which is a consequence of Stein factorization.
ii) A proof is given in \cite{Liu} Chapter 4, Proposition 3.8.
\end{proof}

\begin{lemma} \label{cd-t8}
Assume we are given the following data:\\
i) A finitely presented and proper morphism $\eX \to \spec R$,\\
ii) an open subscheme $\emptyset \neq U \subset \eX_Q$ which is smooth over $Q$, and \\
iii) a connected component $U_0$ of $U$ which is a smooth quasi-projective variety over $Q$.\\
Then there is a proper pair $(\eX_0 , U_0)$ over $R$ together with a finitely presented closed immersion $i : \eX_0 \hookrightarrow \eX$ with the following properties\\
a) $i^{-1} (U) = U \cap \eX_{0Q} = U_0$, and in particular the restriction of $i$ to a morphism $U_0 = i^{-1} (U) \to U$ is finite \'etale.\\
b) The generic fibre $\eX_{0Q}$ is a variety.\\
c) $\eX_0$ is an integral scheme.
\end{lemma}

\begin{rem}
The proposition applies in particular to proper pairs $(\eX , U)$ in which case $U = U_0$ but the more general version will also be needed twice later on.
\end{rem}

\begin{proof}
Using finite presentation and descent arguments together with Lemma \ref{cd-t7} ii) we may assume that $R$ is a discrete valuation ring. Let $X_0$ be the closure of $U_0$ in $X = \eX \otimes_R Q$, i.e. the irreducible component of $X$ containing $U_0$. We give $X_0$ the reduced subscheme structure. Since $U$ is contained in the regular locus of $X^{\red}$ and since the intersection of $X_0$ with another irreducible component of $X$ lies in the singular locus of $X$, it follows that we have $U \cap X_0 = U_0$. Since $U_0$ is a variety, it follows that $X_0$ is a variety as well. The closure $\eX_0$ of $X_0$ in $\eX$ with the reduced structure is a proper $R$-scheme with $\eX_0 \otimes_R Q = X_0$. Since $\eX_0$ is integral by Lemma \ref{cd-t7} ii), and the composition $\eX_0 \hookrightarrow \eX \to \spec R$ is surjective, the morphism $\eX_0 \to \spec R$ is flat.
\end{proof}

We often have to extend morphisms between varieties to models. This is always possible:

\begin{lemma} \label{cd-t9}
With $R$ and $Q$ as above, let $f_Q : Y \to X$ be a morphism of complete varieties over $Q$ and let $\eX$ be a model of $X$ over $R$. Then there exists a model $\Yh$ of $Y$ over $R$ and a proper morphism $f : \Yh \to \eX$ which extends $f_Q$. 
\end{lemma}

\begin{proof}
After descent we may assume that $R$ is a discrete valuation ring. By Nagata's theorem \cite{Luet} Theorem 3.2, there is a proper $R$-scheme $\Yh'$ which contains $Y$ as an open dense subscheme. Since $Y$ is complete, it is open, closed and dense in $\Yh'_Q$ and hence $Y = \Yh'_Q$. By \cite{Luet} Lemma 2.2 there is a blow-up $\Yh$ of $\Yh'$ with $Y =  \Yh'_Q = \Yh_Q$ such that the morphism $f_Q : Y \to X$ extends to a (proper) morphism $f : \Yh \to \eX$. By passing to the irreducible component of $\Yh^{\red}$ containing $Y$ we may assume that $\Yh$ is integral. Since $\Yh$ is proper over $\spec R$ and since the image of $\Yh$ in $\spec R$ contains the generic point it follows that $\lambda : \Yh \to \spec R$ is surjective. Being integral, $\Yh$ is therefore flat over $R$ and hence $\Yh$ is a model of $Y$ over $R$. 
\end{proof}

\begin{defn} \label{cd-t10}
An $R$-scheme $\lambda : \Yh \to \spec R$ is called a good model over $R$ if the following conditions hold:\\
(i) The morphism $\lambda$ is flat, proper and finitely presented (i.e. $\Yh$ is a model).\\
(ii) The generic fibre $\Yh_Q = \Yh \otimes_R Q$ is a (complete) variety.\\
(iii) The special fibre $\Yh_k = \Yh \otimes_R k$ is geometrically reduced.\\
A good  model $\Yh$ is called very good if additionally we have \\
(iv) the morphism $\lambda$ is projective and $\Yh_Q$ is a smooth (projective) variety.
\end{defn}

\begin{prop} \label{cd-t11}
For a good model $\lambda : \Yh \to \spec R$ we have $\lambda_* \Oh_{\Yh} = \Oh_{\spec R}$ universally. The scheme $\Yh$ is integral and $\Yh_k$ is geometrically connected. 
\end{prop}

\begin{proof}
Since $\lambda : \Yh \to \spec R$ is finitely presented, by a descent argument we may assume that $R$ is a discrete valuation ring. By assumption $\lambda$ is flat with geometrically reduced fibres. It follows from \cite{EGA3} Proposition 7.8.6 that the formation of $\lambda_* \Oh_{\Yh}$ commutes with arbitrary base change. It remains to show that $\lambda_* \Oh_{\Yh} = \Oh_{\spec R}$. Since $\lambda$ is proper, the sheaf $\lambda_* \Oh_{\Yh}$ is coherent and given by the finitely generated $R$-module $\Gamma (\spec R , \lambda_* \Oh_{\Yh}) = \Gamma (\Yh , \Oh)$. Note that  $\Yh$ is integral by Lemma \ref{cd-t7} ii), hence this module is torsion free. It is therefore free of finite rank since $R$ is a discrete valuation ring. Thus we have $\lambda_* \Oh_{\Yh} = \Oh^r_{\spec R}$ for some $r \ge 1$. It follows that $\Gamma (\Yh_{\oQu} , \Oh) = \oQu^r$ and hence $r = 1$, since by assumption $\Yh_{\oQu}$ is  integral and therefore connected and reduced. 
By  Lemma \ref{cd-t7} ii), $\Yh$ is integral, and since $\lambda_* \Oh_{\Yh} = \Oh_{\spec R}$ universally, its special fiber is geometrically connected.
\end{proof}

For our construction of the parallel transport we need to dominate proper pairs by good models.

\begin{notation} \label{cd-t12}
Let $\eo_K$ be a Henselian discrete valuation ring of characteristic zero, and let $\eo_{\oK}$ be the normalization of $\eo_K$ in an algebraic closure $\oK$ of $K$ together with the unique valuation prolonging the one on $\eo_K$.
\end{notation}

We then have the following useful fact:

\begin{lemma} \label{cd-t13}
For any proper pair $(\eX, U)$ over $\eo_{\oK}$ there is a very good model $\Yh$ over $\eo_{\oK}$ and a proper morphism $\pi : \Yh \to \eX$ whose restriction $\pi^{-1} (U) \iso U$ is an isomorphism. One can find such a model $\Yh$ of the form $\Yh = \Yh_1 \otimes_{\eo_L} \eo_{\oK}$ where $\Yh_1$ is a very good model over $\eo_L$ for some finite extension $L$ of $K$ which is a normal scheme. 
\end{lemma}

The proof is based on the following results:

\begin{theorem}[Hironaka] \label{cd-t14}
Let $X$ be a projective variety over an algebraically closed field of characteristic zero and let $U \subset X$ be a smooth open subvariety. Then there is a smooth projective variety $Y$ and a proper morphism $\pi : Y \to X$ such that the restriction $\pi^{-1} (U) \iso U$ is an isomorphism. 
\end{theorem}

In \cite{BM} resolution of singularities has been made canonical. It follows that Theorem \ref{cd-t14} also holds in an equivariant setting where a finite group $G$ acts on $X$ preserving $U$. A resolution $\pi : Y \to X$ can then be found where $G$ acts on $Y$ and the morphism $\pi$ is $G$-equivariant. We will need this later on.

Next we need a version of Chow's lemma with control on the center of blow up which follows from \cite{R}, Section 4,  Proposition 5.

\begin{theorem}[Chow's lemma] \label{cd-t15}
Let $S$ be a Noetherian scheme, $\eX$ a proper $S$-scheme and $U \subset \eX$ an open subscheme which is quasi-projective over $S$. Then there is a projective $S$-scheme $\eX'$ with an $S$-morphism $\pi : \eX' \to \eX$ such that the restriction $\pi^{-1} (U) \iso U$ is an isomorphism. 
\end{theorem}

Finally we require a version of Epp's theorem due to T. Saito \cite{Sa} Corollary 1.1.5.

\begin{theorem}[Epp--Saito] \label{cd-t16}
Let $S = \spec \eo_K$ be the spectrum of a discrete valuation ring and $f : \eX \to S$ a normal scheme of finite type with smooth generic fibre. Then there exists a surjection of spectra $S' = \spec \eo_{K'} \to S$ of discrete valuation rings such that $K'$ is a finite extension of $K$ and such that the normalization $\eX'$ of $\eX \times_S S'$ has geometrically reduced fibres over $S'$.
\end{theorem}

\begin{proofof}
{\bf Lemma  \ref{cd-t13}} By Lemma \ref{cd-t8} we may assume that $X = \eX \otimes_{\eo_{\oK}} \oK$ is a variety. Using Theorem \ref{cd-t15} we find a projective variety $Y$ over $\oK$ with a morphism $\pi : Y \to X$ which restricts to an isomorphism $\pi^{-1} (U) \iso U$. By Hironaka's theorem we may assume that in addition $Y$ is smooth. Using finite presentation and possibly replacing $\spec \eo_K$ by a finite extension we may assume that the proper pair $(\eX , U)$ is the base extension of a proper pair $(\eX_0 ,U_0)$ over $\spec \eo_K$. Moreover we may assume that $Y = Y_0 \otimes_K \oK$ and $\pi = \pi_0 \otimes_K \oK$ for a smooth projective variety $Y_0 / K$ and a morphism $\pi_0 : Y_0 \to X_0$ over $\spec K$ which restricts to an isomorphism $\pi^{-1}_0 (U_0) \to U_0$. By Lemma \ref{cd-t9} there is a model $\Yh'_0$ of $Y_0$ and a morphism $\Yh'_0 \to \eX_0$ extending $\pi_0$. Using Chow's lemma alias Theorem \ref{cd-t15} we find a projective scheme $\Yh''_0$ over $\spec \eo_K$ together with a morphism $\Yh''_0 \to \Yh'_0$ which is an isomorphism on the generic fibres, i.e. $\Yh''_0 \otimes_{\eo_K} K = Y_0$. Taking the irreducible component of $\Yh''_0$ containing $Y_0$ we may assume that $\Yh''_0$ is integral. Replacing $\Yh''_0$ by its normalization (which is again projective) in the function field of $\Yh''_0 \otimes K = Y_0$ we may assume that in addition $\Yh''_0$ is normal. By Theorem \ref{cd-t16} there is a finite extension $L$ of $K$ in $\oK$ such that the normalization $\Yh_1$ of $\Yh''_0 \otimes_{\eo_K} \eo_L$ has geometrically reduced fibres over $\spec \eo_L$. Since $\Yh''_0 \otimes_{\eo_K} \eo_L$ is projective over $\eo_L$ and the normalization $\Yh_1$ is finite over $\Yh''_0 \otimes_{\eo_K} \eo_L$, it follows that $\Yh_1$ is projective over $\eo_L$. Moreover $\Yh_1$ is flat over $\spec \eo_L$ since $\Yh_1 \to \spec \eo_L$ is surjective and $\Yh_1$ is integral. Hence $\Yh_1$ is a very good model over $\eo_L$ and $\Yh = \Yh_1 \otimes_{\eo_L} \eo_{\oK}$ is a very good model over $\eo_{\oK}$. By construction the composition
\[
\pi : \Yh \longrightarrow \Yh''_0 \otimes_{\eo_K} \eo_{\oK} \longrightarrow \Yh'_0 \otimes_{\eo_K} \eo_{\oK} \longrightarrow \eX_0 \otimes_{\eo_K} \eo_{\oK} = \eX
\]
induces an isomorphism $\pi^{-1} (U) \to U$.
\end{proofof}

The good models will be used later to construct the parallel transport. For proving independence of the choice of models we will need the following result.

\begin{lemma} \label{cd-t17}
Consider $\eo_{\oK}$ as in Notation \ref{cd-t12}, and let $(\eX , U)$ be a proper pair over $\eo_{\oK}$. Assume that we are given proper morphisms $\pi_i : \Yh_i \to \eX$ with $\pi^{-1}_i (U) \to U$ finite, \'etale  and surjective for $1 \le i \le r$ where $r \ge 1$. Then there exists a very good model $\Yh$ over $\eo_{\oK}$ and a proper morphism $\pi : \Yh \to \eX$ with the following properties for $1 \le i \le r$:\\
(i) $\pi$ factors as a composition $\pi : \Yh \xrightarrow{\rho_i} \Yh_i \xrightarrow{\pi_i} \eX$.\\
(ii) $\pi^{-1} (U) \neq  \emptyset$ and the restriction of each $\rho_i$ to a morphism $\pi^{-1} (U) \to \pi^{-1}_i (U)$ is finite \'etale. 
In particular $\pi^{-1} (U) \to U$ is finite \'etale surjective.
\end{lemma}

\begin{proof}
Set $\Yh' = \Yh_1 \times_{\eX} \ldots \times_{\eX} \Yh_r$ with the canonical proper morphism $\pi' : \Yh' \to \eX$. We have $\pi^{'-1} (U) = \pi^{-1}_1 (U) \times_U \ldots \times_U \pi^{-1}_r (U)$. By assumption, $\pi^{-1}_i (U) \to U$ is surjective for all $i$ and hence $\pi^{'-1} (U) \neq \emptyset$. Moreover the morphism $\pi^{'-1} (U) \to U$ is finite \'etale. Let $U'$ be a connected component of $\pi^{'-1} (U)$. Then $U'$ is a smooth quasi-projective variety over $\oK$. Applying Lemma \ref{cd-t13} to the proper pair $(\Yh' , U')$ we find a very good model $\Yh$ over $\eo_{\oK}$ with a proper morphism $\rho : \Yh \to \Yh'$ such that $\rho^{-1} (U') \iso U'$ is an isomorphism. The inverse images under $\rho$ of the other connected components of $\pi^{'-1} (U)$ are empty since $\Yh_{\oK}$ is irreducible. Hence we have $\rho^{-1} (\pi^{'-1} (U)) = \rho^{-1} (U')$. Consider the compositions $\pi : \Yh \xrightarrow{\rho} \Yh' \xrightarrow{\pi'} \eX$ and $\rho_i : \Yh \xrightarrow{\rho} \Yh' \xrightarrow{\pr_i} \Yh_i$. Properties (i) and (ii) now follow from the construction. 
\end{proof}

For proving that the parallel transport is well defined we will later need an equivariant version of Lemma \ref{cd-t17} in the case where $r = 1$.

\begin{lemma} \label{cd-t18}
Let $(\eX , U)$ be a proper pair over $\eo_{\oK}$ as in Lemma \ref{cd-t17}. Assume that we are given a proper morphism $\pi' : \Yh' \to \eX$ with $V' := \pi^{'-1} (U) \to U$ finite \'etale surjective. Then there exist a finite group $G$, a good model $\Yh$ over $\eo_{\oK}$ with a $G$-action and a $G$-equivariant proper morphism $\pi : \Yh \to \eX$ where $G$ acts trivially on $\eX$ such that the following conditions hold:\\
(i) $\pi$ factors as a composition $\pi : \Yh \xrightarrow{\rho} \Yh' \xrightarrow{\pi'} \eX$.\\
(ii) $V = \rho^{-1} (V')$ is non-empty and the morphism $V \to V'$ and hence $V \to U$ are finite and \'etale. Moreover the morphism $G \times V \to V \times_U V , (g ,v) \mapsto (v , gv)$ is an isomorphism. In particular $G$ acts simply transitively on the geometric fibres of $\pi \, |_V$ over $U$.
\end{lemma}

\begin{proof}
Covering $\Yh'$ by a very good model as in Lemma \ref{cd-t17} for $r = 1$ (or using Lemma \ref{cd-t13}), we may assume that $\Yh'$ is a very good model over $\eo_{\oK}$ and in particular that $\Yh'_{\oK}$ is projective. Choose a finite \'etale morphism of smooth quasi-projective varieties $V \to U$ which factors as $V \to V' \xrightarrow{\pi'} U$ and which is Galois with group $G$ as in (2). Replacing $K$ by a finite extension in $\oK$ we may assume that everything descends to $\eo_K$ or $K$, respectively. Abusing notation, we denote the descended objects by the same letters as before. Next we construct a projective $G$-variety $Y''$ over $Y' = \Yh_K$ containing $V$: Let $Z$ be the relative normalization of $Y'$ in $V$ via the morphism $V \to V' \hookrightarrow Y'$. Then $Z$ is finite over $Y'$ by \cite{stacks} Lemma 28.50.15, and since $Y'$ is projective, the normal variety $Z$ is projective as well. We have a commutative diagram
\[
\begin{array}{ccc}
V & \subset & Z \\
\downarrow & & \downarrow \\
V' & \subset & Y' \; .
\end{array}
\]
For $g \in G$, let $Z_g$ be the relative normalization of $Z$ in $V$ via the morphism $V \xrightarrow{g^{-1}} V \subset Z$. We get a commutative diagram
\[
\xymatrix{
V \ar@{}[r]|{\subset} \ar[d]^{\wr}_{g^{-1}} & Z_g \ar[d]^{\wr} \\
V \ar@{}[r]|{\subset} & Z
}
\]
Let $\varphi_g : Z \iso Z_g$ be the inverse morphism. The group $G$ acts on the fibre product $\prod_{g \in G} Z_g$ over $K$ by the formula
\begin{equation}
\label{eq:2,5}
h \cdot ((z_g)_{g \in G}) = ((\varphi_g \verk \varphi^{-1}_{h^{-1} g}) (z_{h^{-1} g}))_{g \in G} \; .
\end{equation}
The image of $V$ under the diagonal inclusion
\[
V \overset{\Delta}{\hookrightarrow} \prod_{g \in G} Z_g \; , \; x \mapsto (x)_{g \in G}
\]
is $G$-invariant because $h \cdot (x)_{g \in G} = (hx)_{g \in G}$. Let $Y$ be the closure of $\Delta (V)$ in $\prod_{g \in G} Z_g$ with the reduced structure. It is a projective scheme over $K$ with a $G$-action and contains $V \cong \Delta (V)$ as an open $G$-invariant subscheme. $Y$ is integral and since it contains $V$ which is geometrically integral it follows that $Y$ is a (projective) variety over $K$ with a $G$-action extending the one on $V$. We have a commutative diagram where $\pr_e$ is the projection to $Z_e = Z$:
\[
\xymatrix{
V \ar@{}[r]|{\subset} \ar[dd] & Y \ar@{}[r]|{\subset} & \prod_{g \in G} Z_g \ar[d]^{\pr_e} \\ 
 & & Z \ar[d] \\
V' \ar@{^{(}->}[rr] & & Y' \; .
}
\]
We view $Y$ as a $Z$-scheme and hence as a $Y'$ and $X$-scheme. Let $\pi_K : Y \to Y' \to X$ be the morphism to $X$. We have $\pi_K \verk g = \pi_K$ on the open dense subscheme $V$ by $Y$. By separatedness it follows that $\pi_K \verk g = \pi_K$ on $Y$ for all $g \in G$. Next, using equivariant resolution of singularities, we may assume after base change to a finite extension of $K$ that the projective $G$-variety $Y$ is also smooth. Using Lemma \ref{cd-t9} we find a commutative diagram
\[
\xymatrix{
\Zh \ar[d] & Y \ar[l] \ar[d] \\
\Yh' & Y' \ar[l]
}
\]
where $\Zh$ is a model over $\eo_K$ of $Y$ which we may suppose to be normal. For $g \in G$ let $\Zh_g$ be the relative normalization of $\Zh$ in $Y$ via the morphism $Y \xrightarrow{g^{-1}} Y \subset \Zh$. Consider the commutative diagram
\[
\xymatrix{
Y \ar@{}[r]|{\subset} \ar[d]_{g^{-1}} & \Zh_g \ar[d]^{\wr} \\
Y \ar@{}[r]|{\subset} & \Zh
}
\]
and let $\varphi_g : \Zh \iso \Zh_g$ be the inverse morphism. Again we let $G$ act on the fibre product $\prod_{g \in G} \Zh_g$ over $\eo_K$ by formula \eqref{eq:2,5}. The image of $Y$ under the diagonal inclusion $\Delta : Y \hookrightarrow \prod_{g \in G} \Zh_g$ is $G$-invariant and we define $\Yh$ to be the closure of $\Delta (Y)$ in $\prod_{g \in G} \Zh_g$ with the reduced subscheme structure. It is an integral proper $\eo_K$-scheme with generic fibre $Y$ and a $G$-action extending the one on $Y$. We have a commutative diagram where $\pr_e$ is the projection to $\Zh_e = \Zh$
\[
\xymatrix{
Y \ar@{}[r]|{\subset} \ar[dd] & \Yh \ar@{}[r]|{\subset} & \prod_{g \in G} \Zh_g \ar[d]^{\pr_e} \\
 & & \Zh \ar[d] \\
 Y' \ar@{^{(}->}[rr] & & \Yh'\; .
}
\]
Thus $\Yh$ becomes a $\Yh'$ and  hence an $\eX$-scheme, and the generic fibre of the morphism $\pi : \Yh \to \Yh' \to \eX$ is the morphism $\pi_K$ defined above. As before $\pi \verk g = \pi$ on the open dense subscheme $Y$ of $\Yh$ and separatedness implies that we have $\pi \verk g = \pi$ on all of $\Yh$. By passing to the normalization we may assume that $\Yh$ is normal. We now use Theorem \ref{cd-t16}: Replacing $\Yh$ by the normalization of $\Yh \otimes_{\eo_K} \eo_L$ for a suitable finite extension $L$ of $K$ we obtain a good model over $\eo_L$ of $Y \otimes_K L$ mapping to $\Yh' \otimes_{\eo_K} \eo_L$ and $\eX \otimes_{\eo_K} \eo_L$. Since normalization is functorial the $G$-action extends to the new $\Yh$. Base changing from $\eo_L$ to $\eo_{\oK}$ we are done.
\end{proof}

\begin{rems}
a) The model $\Yh$ that we constructed is actually projective over $\eo_K$ and hence a very good model (since $\Yh_{\oK}$ is smooth by construction). We do not need this however.\\
b) A similar method to extend $G$-actions is used in the proof of Lemma 2.4 in \cite{Liu2}.
\end{rems}
\section{Parallel transport modulo $ p^n$} \label{sec:cd3}

In this section, we define parallel transport modulo $p^n$ as a functor on the fundamental groupoid. Fix an algebraic closure $\oQ_p$ of $\Q_p$ and let $\C_p$ be the completion of $\oQ_p$. By $\oZ_p$ and $\eo= \eo_{\C_p}$ we denote the corresponding rings of integers. 

The version of the \'etale fundamental groupoid that we will use is the following. Let $Z$ be a variety over $\oQ_p$ with a point $z \in Z (\C_p)$. The fiber functor $F_z$ from the category of finite \'etale covers $\tZ$ of $Z$ to the category of finite sets attaches to $\tZ$ the set of $\C_p$-valued points of $\tZ$ over $z$. The objects of the (topological) category $\Pi_1 (Z)$ are the points of $Z (\C_p)$. For $z , z' \in Z (\C_p)$ the morphism space
\[
\Mor_{\Pi_1 (Z)} (z, z') = \Iso (F_z , F_{z'})
\]
is pro-discrete, hence totally disconnected and a compact Hausdorff space. Composition of morphisms is continuous. A morphism $\gamma$ from $z$ to $z'$ is called an \'etale path (up to homotopy). The Grothendieck fundamental group of $Z$ with base point $z$ is the automorphism group of the fiber functor $F_z$:
\[
\pi_1 (Z,z) = \Mor_{\Pi_1 (Z)} (z,z) = \Aut (F_z) \; .
\]
For any morphism $f : Z \to Z'$ of varieties over $\oQ_p$ there is a natural continuous functor $f_* : \Pi_1 (Z) \to \Pi_1 (Z')$. Here a functor between topological categories is continuous if the maps between the topological morphism spaces are continuous. Any automorphism $\sigma$ of $\oQ_p$ over $\Q_p$ induces a continuous automorphism $\sigma_* : \Pi_1 (Z) \to \Pi_1 ({}^{\sigma}\! Z)$ where $^{\sigma} Z = Z \otimes_{\oQ_{p,\sigma}} \oQ_p$. These functorialities are explained in detail in \cite{dewe1} \S\,3, p. 578/579.

A groupoid is a category all of whose morphisms are isomorphisms. We call a groupoid $\Ch$ quasi-finite if $\Mor (C,C') \neq \emptyset$ holds for all $C , C'$ in $\Ch$ and if for one (and hence every) object $C$ every finitely generated subgroup of $\Mor (C,C) = \Aut (C)$ is finite. For any element $\omega \in \oZ_p$ with $0 < |\omega| < 1$ the ring $\eo / \omega \eo = \oZ_p / \omega \oZ_p$ is the filtered union of finite rings. It follows that any finitely generated subgroup of $\GL_r (\eo / \omega \eo)$ is finite. The groupoid $\Free_r (\eo / \omega)$ of free $\eo / \omega$-modules of rank $r$ is therefore quasi-finite. We can now formulate the Seifert-van Kampen theorem that we need:

\begin{theorem} \label{cd-t24}
i) Let $Z$ be a normal variety over $\oQ_p$ and let $\gamma$ be an \'etale path from $x \in Z (\C_p)$ to $x' \in Z (\C_p)$ in $Z$. Then for all open subschemes $U$ and $U'$ of $Z$ with $x \in U (\C_p)$ and $x' \in U' (\C_p)$ and any point $y \in (U \cap U') (\C_p)$, there are \'etale paths $\gamma_U$ from $x$ to $y$ in $U$ and $\gamma_{U'}$ from $y$ to $x'$ in $U'$ such that we have
\[
\gamma = (j_{U'})_* \gamma_{U'} \verk (j_U)_* \gamma_U \; .
\]
ii) Let $Z$ be a variety over $\oQ_p$ and let $\eU$ be an open covering of $Z$ which is stable under finite intersections. Let $\Ch$ be a quasi-finite groupoid and assume that we are given functors $\rho_U : \Pi_1 (U) \to \Ch$ for every member $U$ of the covering $\eU$ such that for every inclusion $j : V \hookrightarrow U$ we have $\rho_V = \rho_U \verk j_*$. Then there is a unique functor $\rho : \Pi_1 (Z) \to \Ch$ with $\rho_U = \rho \verk j_{U*}$ for every $U$ in $\eU$ with $j_U : U \hookrightarrow X$ denoting the inclusion. 
\end{theorem}

\begin{proof}
i) Since connected schemes are \'etale pathwise connected, see \cite{SGA1} Expos\'e V, \S\,7, we can choose an \'etale path $\delta$ in $Z$ from $x'$ to $y$. Then $\gamma_1 = \delta \verk \gamma$ and $\gamma_2 = \delta^{-1}$ are \'etale paths in $Z$ from $x$ to $y$ and from $y$ to $x'$, respectively, and we have $\gamma = \gamma_2 \verk \gamma_1$. It suffices to show that
\[
(j_U)_* : \Mor_{\Pi_1 (U)} (x,y) \longrightarrow \Mor_{\Pi_1 (Z)} (x,y) 
\]
is surjective, and that the analogous map $(j_{U'})_*$ is surjective. Since $U$ is connected and hence \'etale path connected, it suffices to show that the homomorphism
\[
(j_U)_* : \pi_1 (U,x) \longrightarrow \pi_1 (Z,x)
\]
is surjective. Let $\Kh$ be the common function field of $U$ and $Z$ and $\oKh$ an algebraic closure. Set $\oeta = \spec \oKh$. Connecting $x$ to $\oeta$ in $U$ (and hence in $Z$) it suffices to show that the map
\[
(j_U)_* : \pi_1 (U, \oeta) \longrightarrow \pi_1 (Z, \oeta)
\]
is surjective. This follows immediately from the description of $\pi_1 (U, \oeta)$ and  $\pi_1 (Z,\oeta)$ as the Galois groups over $\Kh$ of the maximal over $U$ and over $Z$, respectively,  unramified extension field of $\Kh$ in $\oKh$. Here we need that $Z$ and hence $U$ are normal varieties, see \cite{SGA1}, Expos\'e V, Proposition 8.2. \\
ii) Although i) provides part of an algebraic proof of ii) we do not know a quick way to give a full algebraic proof of ii). We therefore reduce ii) to a known Seifert-van Kampen theorem in topology using several standard results. Note that the algebraic fundamental group of a variety over an algebraically closed field of characteristic zero does not change under base extension to a bigger algebraically closed field. It follows that $\Pi_1 (Z) = \Pi_1 (Z_{\C_p})$ where the objects are the same, i.e. $Z (\C_p)$, and the morphisms of $\Pi_1 (Z_{\C_p})$ are defined using fibre functors on the category of finite \'etale coverings of $Z_{\C_p}$. Next we identify $\C_p$ with $\C$ as fields and claim that
\begin{equation}
\label{eq:cd3}
\widehat{\Pi}^{\ttop}_1 (Z (\C)) = \Pi_1 (Z_{\C_p}) \; .
\end{equation}
Here the category $\widehat{\Pi}^{\ttop}_1 (Z (\C))$ has the same object set $Z (\C)$ as $\Pi^{\ttop}_1 (Z (\C))$. For $z , z'$ in $Z (\C)$ the homotopy classes of paths from $z$ to $z'$ form the morphism set $\Mor (z,z')$ of $\Pi^{\ttop}_1 (Z (\C))$. The group 
\[
\Gamma = \pi^{\ttop}_1 (Z (\C) , z) = \Mor (z,z)
\]
acts simply transitively on $\Mor (z,z')$. The morphisms from $z$ to $z'$ in $\widehat{\Pi}^{\ttop}_1 (Z (\C))$ are given by the profinite set
\[
\widehat{\Mor} (z,z') = \varprojlim_N \Mor (z, z') / N \; .
\]
Here $N$ runs over the normal subgroups of $\Gamma$ of finite index. The analogous construction using the right action by $\Gamma' = \Mor (z',z')$ gives the same result. 
Equation (\ref{eq:cd3}) now follows from the natural isomorphisms $\hat{\pi}^{\ttop}_1 (Z (\C) , z) = \pi_1 (Z_{\C_p} , z)$ for $z \in Z (\C_p) = Z (\C)$, see  \cite{SGA1} XII, Cor. 5.2, since the objects of both categories coincide by our identification of  $\C_p$ with $\C$.

According to \cite{May} Ch. II, \S\,7 we have
\begin{equation}
\label{eq:cd4}
\Pi^{\ttop}_1 (Z (\C)) = \varinjlim_{\eU} \Pi^{\ttop}_1 (U (\C)) \; .
\end{equation}
Using the functors
\[
\Pi^{\ttop}_1 (U (\C)) \longrightarrow \widehat{\Pi}^{\ttop}_1 (U (\C)) = \Pi_1 (U) \quad \text{for} \; U \in \eU \; ,
\]
the functors $\rho_U : \Pi_1 (U) \to \Ch$ define functors $\rho^{\ttop}_U : \Pi^{\ttop}_1 (U (\C)) \to \Ch$ which are compatible with morphisms in $\eU$, i.e. with the inclusions $V \hookrightarrow U$ for $U,V \in \eU$. Using \eqref{eq:cd4} we obtain a unique functor
\[
\rho^{\ttop} : \Pi^{\ttop}_1 (Z (\C)) \longrightarrow \Ch
\]
which is compatible with all $\rho^{\ttop}_U$. The fundamental group $\pi^{\ttop}_1 (Z (\C) , z)$ is finitely generated since $Z$ is a quasi-projective variety. Since $\Ch$ is quasi-finite, the image of the group $\Mor (z,z) = \Pi^{\ttop}_1 (Z (\C) , z)$ in $\Mor_{\Ch} (\rho (z) , \rho (z)) = \Aut_{\Ch} (\rho (z))$ is finite and we get a factorization
\[
\Mor (z,z) \longrightarrow \widehat{\Mor} (z,z) \longrightarrow \Mor_{\Ch} (\rho (z) , \rho (z)) \; .
\]
It follows that for all $z , z' \in Z (\C) = Z (\C_p)$ we have a factorization
\[
\Mor (z,z') \longrightarrow \widehat{\Mor} (z,z') \longrightarrow \Mor_{\Ch} (\rho (z) , \rho (z')) \; .
\]
Hence the functor $\rho^{\ttop}$ factors:
\[
\rho^{\ttop} : \Pi^{\ttop}_1 (Z (\C)) \longrightarrow \widehat{\Pi}^{\ttop}_1 (Z (\C)) = \Pi_1 (Z) \xrightarrow{\rho} \Ch \; .
\]
The desired property of $\rho$ and its uniqueness follow with little effort from the corresponding assertions for $\rho^{\ttop}$, again using that $\Ch$ is quasi-finite. 
\end{proof}

\begin{notation}\label{not}
Recall that $\oZ_p$ and $\eo $ denote the valuation rings in $\oQ_p$ and $\C_p$, respectively. By $k = \oF_p$ we denote their common residue field.

For all $n \geq 1$ we set $\eo_n = \eo / p^n \eo = \oZ_p / p^n \oZ_p$. If $\eX$ is a scheme over $\eo$ or $\oZ_p$, and if $\Eh$ is a vector bundle on $\eX$, we denote by  $\eX_n$ and $\Eh_n$  the base changes of $\eX$ and $ \Eh$, respectively,  with $\eo_n$. Similarly, if $\pi$ is a morphism between $\eo$- or $\oZ_p$-schemes, we denote its base change with $\eo_n$ by $\pi_n$. 
\end{notation}

\begin{defn} \label{cd-t25}
Let $(\eX , U)$ be a proper pair over $\Z_p$ and $\Eh$ a vector bundle on $\eX_{\eo}$. A proper morphism $\pi : \Yh \to \eX$ is called a (very) good trivializing cover for $\Eh_n$ over $(\eX , U)$ if $\Yh$ is a (very) good model over $\oZ_p$, the restriction $\pi^{-1} (U) \to U$ is finite \'etale surjective and $\pi^*_n \Eh_n$ is a trivial bundle. 
\end{defn}

The parallel transport modulo $p^n$  will be obtained by glueing instances of the following basic construction.

\begin{constr} \label{cd-t26}
Let $(\eX , U)$ be a proper pair over $\oZ_p$ and $\Eh$ a vector bundle of rank $r$ over $\eX$. Assuming that there is a good trivializing cover $\pi : \Yh \to \eX$ for $\Eh_n$ over $(\eX , U)$ we now construct a functor
\[
\rho_{\Eh , n} (U) : \Pi_1 (U) \longrightarrow \Free_r (\eo_n) \; 
\]
in the following way:
\end{constr}

Any point $x \in U(\C_p)$ defines by properness a section $x_{\eo} \in \Xh_{\eo}(\eo)$, and hence a point $x_n \in \Xh_{\eo}(\eo_n)$ after reduction modulo $p^n$. On objects, we associate to every $x \in U(\C_p)$ the fiber $\Eh_{x_n} = x_n^\ast \Eh$, viewed as a free $\eo_n$-module. 

Let $\gamma$ be an \'etale path from $x$ to $x'$ in $U$. We fix a preimage $y \in \pi^{-1}(x)$ of $x$. Since $\pi^{-1}(U) \rightarrow U$ is a finite \'etale covering, $\gamma$ gives rise to a point $\gamma y \in \pi^{-1}(x')$. By properness, $x'$ extends to a section $x'_{\eo}$ in $\Xh_\eo(\eo)$, which can be reduced modulo $p^n$ to $x'_n \in \Xh_\eo (\eo_n)$. Similarly, $y$ and $\gamma y$ give rise to points $y_n$ and $(\gamma y)_n$ in $\Yh_\eo (\eo_n)$. 

Let $\lambda : \Yh \to \spec \oZ_p$ be the structural morphism and $\lambda_n : \Yh_n \to \spec \eo_n$ its reduction $\mod p^n$. Since $\lambda_* \Oh_{\Yh} = \Oh_{\spec \oZ_p}$ holds universally by Proposition \ref{cd-t11}, we have $\lambda_{n*} \Oh_{\Yh_n} = \Oh_{\spec \eo_n}$. Since $\pi_n^* \Eh_n$ is a trivial bundle, pullback by $y_n : \spec \eo_n \to \Yh_n$ induces an isomorphism
\[
y^*_n : \Gamma (\Yh_n , \pi^*_n \Eh_n) \iso \Gamma (\spec \eo_n , y^*_n \pi^*_n \Eh_n) = \Gamma (\spec \eo_n , x^*_n \Eh_n) = \Eh_{x_n} \; ,
\]
and similarly for $(\gamma y)^*_n$. We now define $\rho_{\Eh,n} (U) (\gamma)$ as the composition
\[
\rho_{\Eh,n} (U) (\gamma) = (\gamma y)^*_n \verk (y^*_n)^{-1} : \Eh_{x_n} \longrightarrow \Eh_{x'_n} \; .
\]

\begin{lemma} \label{cd-t27} In the situation of Construction \ref{cd-t26}, 
the functor $\rho_{\Eh,n} (U) : \Pi_1 (U) \to \Free_r (\eo_n)$ depends only on $(\eX,U)$ and $\Eh$ and not on the auxiliary choices. 

If $(\teX , \tU)$ is another proper pair and $\alpha : (\teX, \tU) \to (\eX,U)$ a morphism of proper pairs,  there exists also a good trivializing cover for $(\alpha^* \Eh)_n$ over $(\teX , \tU)$ and hence $\rho_{\alpha^* \Eh , n} (\tU)$ is defined. The following formula holds:
\begin{equation} \label{eq:cd5}
\rho_{\Eh,n} (U) \verk \alpha_* = \rho_{\alpha^* \Eh,n} (\tU) : \Pi_1 (\tU) \longrightarrow \Free_r (\eo_n) \; .
\end{equation}
Here for simplicity, we have written $\alpha_*$ for $(\alpha \; |_{\tU})_*$ where $\alpha \, |_{\tU} : \tU \to U$ is the induced morphism.
\end{lemma}

\begin{proof}
Let $\tpi : \tYh \to \eX$ be another good trivializing cover for $\Eh_n$ of $(\eX , U)$ and assume that $\tpi$ factors as $\tpi : \tYh \xrightarrow{\rho} \Yh \xrightarrow{\pi} \eX$ where $\rho$ is proper and the maps $\tpi^{-1} (U) \xrightarrow{\rho} \pi^{-1} (U) \xrightarrow{\pi} U$ are finite and \'etale. Choose a preimage $\ty \in \tpi^{-1} (x)$ of $x$. An \'etale path $\gamma$ from $x$ to $x'$ in $X$ gives rise to a preimage $\gamma \ty \in \tpi^{-1} (x')$ of $x'$. Then $y = \rho (\ty)$ and $\gamma y = \rho (\gamma \ty)$ are preimages in $\pi^{-1} (U)$ of $x$ and $x'$ as in construction \ref{cd-t26}. Using the commutative diagram
\[
\xymatrix{
\Eh_{x'_n} \ar@{=}[d] & \Gamma (\Yh_n , \pi^*_n \Eh_n) \ar[l]_-{\overset{(\gamma y)^*_n}{\sim}} \ar[d]^{\rho^*_n} \ar[r]^-{\overset{y^*_n}{\sim}} & \Eh_{x_n} \ar@{=}[d] \\
\Eh_{x'_n} & \Gamma (\tYh_n , \tpi^*_n \Eh_n) \ar[l]_-{\overset{(\gamma \ty)^*_n}{\sim}} \ar[r]^-{\overset{\ty^*_n}{\sim}} & \Eh_{x_n} 
}
\]
we find that
\begin{equation}\label{eq:cd6}
(\gamma y)^*_n \verk (y^*_n)^{-1} = (\gamma \ty)^*_n \verk (\ty^*_n)^{-1} \; .
\end{equation}
We can now show that the definition of the parallel transport in Construction \ref{cd-t26} is independent of the choice of a preimage $y \in \pi^{-1} (x)$ of $x \in U (\C_p)$. Because of Lemma \ref{cd-t18}, the connectedness of $\pi^{-1} (U)$ and formula \eqref{eq:cd6} we may assume the following:\\
1) there is a finite group $G$ acting on $\Yh$ such that $\pi : \Yh \to \eX$ is equivariant if $G$ acts trivially on $\eX$.\\
2) The group $G$ acts (simply) transitively on $\pi^{-1} (x)$ for $x \in U (\C_p)$.\\
Given another preimage $z \in \pi^{-1} (x)$ there is a (unique) automorphism $\sigma \in G$ of $\Yh$ such that $z =\sigma y$. Hence we have $\gamma (z) = \sigma (\gamma (y)) = \sigma \verk \gamma (y)$. It follows that
\[
z^*_n = y^*_n \verk \sigma^* : \Gamma (\Yh_n , \pi^*_n \Eh_n) \xrightarrow{\overset{\sigma^*}{\sim}} \Gamma (\Yh_n , \pi^*_n \Eh_n) \xrightarrow{\overset{y^*_n}{\sim}} \Eh_{x_n}
\]
and similarly that $(\gamma z)^*_n = (\gamma y)^*_n \verk \sigma^*$. This gives
\[
(\gamma z)^*_n \verk (z^*_n)^{-1} = (\gamma y)^*_n \verk \sigma^* \verk (\sigma^*)^{-1} \verk (y^*_n)^{-1} = (\gamma y)^*_n \verk (y^*_n)^{-1} \; .
\]
By Lemma \ref{cd-t17} any two proper morphisms $\pi_i : \Yh_i \to \eX$ for $i = 1,2$ from good models $\Yh_i$ over $\oZ_p$ such that $\pi^{-1}_i (U) \to U$ is finite \'etale surjective are dominated by a third such morphism $\tpi : \tYh \to \eX$ with analogous properties. Using the previous arguments and formula \eqref{eq:cd6}, it follows that the parallel transport $\rho_{\Eh,n}$ is also independent of the choice of a good trivializing cover $\pi : \Yh \to \eX$ in its definition.

Given a morphism $\alpha : (\teX , \tU) \to (\eX , U)$ of proper pairs choose a good trivializing cover  $\pi : \Yh \to \eX$ for $\Eh$ as in Construction \ref{cd-t26}. Then the proper morphism $\opi : \alpha^{-1} (\Yh) = \teX \times_{\eX} \Yh   \to \teX$   is finite \'etale and surjective over $\tU$ and trivializes $\alpha^* \Eh$  modulo $ p^n$. Choose a connected component $\tV$ of the smooth quasiprojective scheme $\opi^{-1} (\tU)$. Since $\tU$ is connected and non-empty the finite \'etale morphism $\tV \to \tU$ is surjective. Using Lemmas \ref{cd-t8} and \ref{cd-t13} we find a good model $\tYh$ and a proper morphism $\rho : \tYh \to \alpha^{-1} (\Yh)$ such that the composition $\tpi : \tYh \xrightarrow{\rho} \alpha^{-1} (\Yh) \xrightarrow{\opi} \oeX$ restricts to a finite \'etale, surjective morphism $\tpi^{-1} (\tU) \cong \tV \to \tU$. It is clear that formula \eqref{eq:cd5} holds on objects. To check it on morphisms choose a path $\tgamma$ from $\tx \in \tU (\C_p)$ to $\tx' \in \tU (\C_p)$. Then $\gamma = \alpha_* (\tgamma)$ is a path from $x = \alpha (\tx)$ to $x' = \alpha (\tx')$. Consider the commutative diagram
\[
\xymatrix{
\tYh \ar[d]_{\rho} \ar[dr]^{\talpha} \\
\alpha^{-1} (\Yh) \ar[r] \ar[d]_{\opi} & \Yh \ar[d]^{\pi} \\
\teX \ar[r]^{\alpha} & \eX
}
\]
where $\rho^{-1} (V) \cong V \subset \alpha^{-1} (\pi^{-1} (U))$. We have to show that the diagram
\[
\xymatrix{
\alpha^* (\Eh)_{\tx_n} \ar@{=}[d] \ar[rr]^{\rho_{\alpha^* \Eh , n} (\tU) (\tgamma)} & & \alpha^* (\Eh)_{\tx'_n} \ar@{=}[d] \\
\Eh_{x_n} \ar[rr]^{\rho_{\Eh,n} (U) (\gamma)} & & \Eh_{x'_n}
}
\]
commutes. Choose a point $\ty \in \tpi^{-1} (\tx)$ and let $y = \talpha (\ty) = \talpha \verk \ty$ be its image in $\pi^{-1} (x)$. By definition
\[
\gamma y = \alpha_* (\tgamma) y = \talpha (\tgamma \ty) = \talpha \verk \tgamma \ty \; .
\]
Hence we have:
\begin{align*}
\rho_{\Eh,n} (U) (\gamma) & = (\gamma y)^*_n \verk (y^*_n)^{-1} = (\tgamma \ty)^*_n \verk \talpha^*_n \verk (\talpha^*_n)^{-1} \verk (\ty^*_n)^{-1} \\
& = (\tgamma \ty)^*_n \verk (\ty^*_n)^{-1} = \rho_{\alpha^* \Eh , n} (\tU) (\tgamma) \; .
\end{align*}
Here $\talpha^*_n$ refers to the isomorphism:
\[
\talpha^*_n : \Gamma (\Yh_n , \pi^*_n \Eh_n) \iso \Gamma (\tYh_n , \talpha^*_n \pi^*_n \Eh_n) = \Gamma (\tYh_n , \tpi^*_n \alpha^*_n \Eh_n) \; .
\]
Note here that $\pi^*_n \Eh_n$ is a trivial vector bundle and that $\Gamma (\Yh_n , \Oh) = \eo_n = \Gamma (\tYh_n , \Oh)$ since both $\Yh$ and $\tYh$ are good models over $\oZ_p$. 
\end{proof}
\section{Trivializing covers for bundles with numerically flat reduction}\label{section5}
We sill now introduce the category of vector bundles for which we can define parallel transport. 
\begin{defn}\label{cd-t19}
Let $X$ be a smooth, complete variety over $\oQ_p$ and $\eX$ a model of $X$ over $\oZ_p$. We denote by $\Bh^s (\eX_{\eo})$ the full subcategory of vector bundles $\Eh$ on $\eX_{\eo}$ whose special fibre $\Eh_k$ is a numerically flat bundle on the complete connected $k$-scheme $\eX_k$. 
\end{defn}

Note here that $\eX_k$ is connected by Lemma \ref{cd-t7}.

\begin{prop} \label{cd-t20}
The $\eo$-linear exact category $\Bh^s (\eX_{\eo})$ is closed under extensions, tensor products, duals, internal homs and exterior powers.
\end{prop}

\begin{proof}
This is a consequence of the corresponding closure properties for Nori-semistable bundles on smooth projective curves over a field which are well known. Incidentally, in our situation where the base field is $k = \oF_p$ they follow without effort from Theorem \ref{cd-t2}.
\end{proof}

In order to construct  parallel transport $\mod p^n$ for the bundles $\Eh$ in $\Bh^s (\eX_{\eo})$ we need to know that they have suitable trivializing covers so that we can apply Construction \ref{cd-t26}. The following definition makes this precise. 

\begin{defn} \label{cd-t21}
Let $X$ be a smooth, complete variety over $\oQ_p$ and $\eX$ a model of $X$ over $\oZ_p$. For an ideal $\en$ in $\oZ_p$ let $\Bh_{\en} (\eX_{\eo})$ be the full subcategory of vector bundles $\Eh$ on $\eX_{\eo}$ for which the following holds:\\
i) There is an open covering $\{ U_i \}$ of $X$ and a family $\{ \Yh_i \}$ of good (or by Lemma \ref{cd-t17} even very good) models over $\oZ_p$ together with proper morphisms $\pi_i : \Yh_i \to \eX$ such that $\pi^{-1}_i (U_i) \to U_i$ is finite \'etale surjective for all $i$.\\
ii) The pullback via $\pi_i$ of $\Eh \otimes_{\eo} \eo / \en$ to $\Yh_{i \eo} \otimes_{\eo} \eo / \en = \Yh_i \otimes_{\oZ_p} \oZ_p / \en$ is a trivial bundle for each $i$.
\end{defn}

If $\en = \omega \oZ_p$ for some $\omega \in \oZ_p$ with $0 < |\omega| < 1$ we write $\Bh_{\omega} (\eX_{\eo})$ for $\Bh_{\en} (\eX_{\eo})$. 

\begin{theorem} \label{cd-t22} Let $X$ be a smooth, complete variety over $\oQ_p$ and $\eX$ a model of $X$ over $\oZ_p$. We denote by $\emm$ the maximal ideal in $\oZ_p$. Then 
$\Bh^s (\eX_{\eo}) = \Bh_\emm(\eX_{\eo})$. Hence every vector bundle on $\eX_\eo$ with numerically flat reduction can be trivialized on the special fibers of a covering  $\{ \Yh_i \}$ as in Definition \ref{cd-t21} i).
\end{theorem}

\begin{proof}
By Lemma \ref{cd-t7} the special fibre $\eX_k$ of $\eX$ is connected and $\eX$ is integral with $X$ open and dense in $\eX$.

For a vector bundle $\Eh$ in $\Bh_\emm (\eX_{\eo})$ there exist a good model $\Yh$ over $\oZ_p$ and a proper morphism $\pi : \Yh \to \eX$ such that $\pi (\Yh)$ contains an open subset of $\Xh$ and such that $\pi^*_k \Eh_k$ is a trivial bundle on the complete connected fibre $\Yh_k$. In particular $\pi^*_k \Eh_k$ is numerically flat. The image $\pi (\Yh)$ is closed in $\eX$ and contains the generic point of $\eX$. Hence $\pi : \Yh \to \eX$ and therefore also $\pi_k : \Yh_k \to \eX_k$ are surjective. Since $\pi^*_k \Eh_k$ is numerically flat, it follows that $\Eh_k$ is numerically flat as well. Hence $\Bh_\emm (\eX_{\eo})$ is a full subcategory of $\Bh^s (\eX_{\eo})$.

 For the proof of the reverse inclusion $\Bh^s (\eX_{\eo}) \subset \Bh_\emm (\eX_{\eo})$ a simple argument using Lemma \ref{cd-t13} shows that in addition we may assume that $X / \oQ_p$ is projective and that $\eX$ is a very good model: Cover $X$ by smooth quasiprojective open subvarieties $U$ and consider the very good models $\Yh$ from Lemma \ref{cd-t13} for the proper pairs $(\eX , U)$. Use the inclusions $\Bh^s (\Yh_{\eo}) \subset \Bh_\emm (\Yh_{\eo})$ for the models $\Yh$ to get trivializing covers over all $\Yh$'s and hence over $\eX$. Thus let $\Eh$ be a vector bundle in $\Bh^s (\eX_{\eo})$ for a very good model $\eX$ over $\oZ_p$ of a smooth projective variety $X$ over $\oQ_p$. Since the restriction $\Eh_k$ of $\Eh$ to the projective, reduced connected special fibre $\eX_k$ is numerically flat by assumption, Theorem \ref{cd-t2} gives us the following data:

A projective connected scheme $\Yh_k$ over $k$,  a finite \'etale covering $\pi_k : \Yh_k \to \eX_k$ and an integer $\nu \ge 0$ such that under the composition
\[
\varphi : \Yh_k \xrightarrow{F^{\nu}} \Yh_k \xrightarrow{\pi_k} \eX_k
\]
the pullback of $\varphi^* \Eh_k$ is a trivial bundle. Here $F$ is the Frobenius $ \Fr_{\kappa} \otimes_{\kappa} \id_k$ corresponding to an isomorphism $\Yh_k = \Yh_{\kappa} \otimes_{\kappa} k$ where $\Yh_{\kappa}$ is a projective geometrically connected scheme over a finite field $\kappa \subset k$ and $\Fr_{\kappa} = \Fr^r_p$ with $r = [\kappa : \F_p]$. By enlarging $\kappa$ if necessary, we may assume that $\eX$ descends to a projective model $\eX_{\eo_K}$ over the ring of integers $\eo_K$ in a finite extension $K$ of $\Q_p$ with residue field $\kappa$. Additionally we may arrange for $\pi_k$ to be the base extension of a finite \'etale morphism $\pi_{\kappa} : \Yh_{\kappa} \to \eX_{\kappa} := \eX_{\eo_{\kappa}} \otimes_{\eo_K} \kappa$. Using \cite{SGA1} IV, Th\'eor\`eme 1.10 we may lift $\pi_{\kappa} : \Yh_{\kappa} \to \eX_{\kappa}$ to a finite \'etale morphism $\pi_{\eo_K} : \Yh_{\eo_K} \to \eX_{\eo_K}$ whose special fibre is $\pi_{\kappa}$. The scheme $\Yh_{\eo_K}$ is projective over $\eo_K$ since $\eX_{\eo_K}$ has this property and since $\pi_{\eo_K}$ is finite. The generic fibre $X_K$ of $\eX_{\eo_K}$ is a smooth projective variety over $K$. Hence the generic fibre $Y_K$ of $\Yh_{\eo_K}$ is a smooth projective scheme over $K$ and in particular geometrically reduced. Since $\Yh_k$ is connected, projective and reduced (since $\eX_k$ is reduced and $\pi_k$ is \'etale), we have $\dim_k H^0 (\Yh_k , \Oh)= 1$. By semi-continuity, setting $Y = Y_K \otimes_K \oQ_p$ we obtain $\dim_{\oQ_p} H^0 (Y, \Oh) = 1$ and hence that $Y$ is connected. Thus $Y$ is a smooth projective variety over $\oQ_p$. We conclude that $\Yh_{\eo_K}$ and $\Yh := \Yh_{\eo_K} \otimes_{\eo_K} \oZ_p$ are very good models over $\eo_K$ and over  $\oZ_p$, respectively. Let $\pi = \pi_{\eo_K} \otimes \oZ_p : \Yh \to \eX$ be the base extension of $\pi_{\eo_K}$. We want to prove that our bundle $\Eh$ in $\Bh^s (\eX_{\eo})$ lies in $\Bh_k (\eX_{\eo})$. For this we have to show the following: For any point $x \in X (\oQ_p)$ there exist an open subscheme $U \subset X$ with $x \in U (\oQ_p)$ and a good model $\Zh$ over $\oZ_p$ with a proper morphism $\mu : \Zh \to \eX$ such that $\mu^{-1} (U) \to U$ is finite \'etale surjective and such that the bundle $\mu^*_k \Eh_k$ is a trivial bundle on $\Zh_k$. 

We have lifted $\pi_k : \Yh_k \to \eX_k$ to a proper morphism $\pi : \Yh \to \eX$ where $\pi = \pi_{\eo_K} \otimes_{\eo_K} \oZ_p$. However in general $F^{\nu}$ cannot be lifted to a morphism $F^{\nu} : \Yh \to \Yh$. Fortunately it can be lifted to a morphism $F^{\nu} : \Yh' \to \Yh$ where $(\Yh'_k)^{\red} \cong \Yh_k$ and this is enough for our purposes as we will see in a moment.

We need the following Lemma, whose proof will be given at the end of this section.

\begin{lemma} \label{cd-t39}
For a finite extension $K$ of $\Q_p$ with ring of integers $\eo_K$ consider a closed $\eo_K$-subscheme $\Wh$ of $\Pa^N_{\eo_K}$ together with finitely many points $y_1 , \ldots , y_r \in \Wh (\oK)$. Let $T = \Ge^N_{m,K}$ be the open subscheme of $\Pa^N_K$ where all coordinates are non-zero. Then there is a closed immersion $\Wh \hookrightarrow \Pa^N_{\eo_K}$ over $\eo_K$ such that $y_1 , \ldots , y_r \in \Wh (\oK) \cap  T (\oK)$. 
\end{lemma}

Applying the lemma to the projective $\eo_K$-scheme $\Wh = \Yh_{\eo_K}$ and to $\{ y_1 , \ldots , y_r \} = \pi^{-1} (x) \subset \Yh (\oK) = Y (\oK)$ we find a closed immersion $\Yh_{\eo_K} \subset \Pa^N_{\eo_K}$ for some $N \ge 1$ which we view as an inclusion such that $\pi^{-1} (x) \subset T (\oK)$. Define $\Yh'_{\eo_K}$ by the cartesian diagram
\[
\xymatrix{
\Yh'_{\eo_K} \ar@{^{(}->}[r] \ar[d]_{\tau_{\,\eo_K}} & \Pa^N_{\eo_K} \ar[d]^{F^{\nu}_{\eo_K}} \\
\Yh_{\eo_K} \ar@{^{(}->}[r]^{} & \Pa^N_{\eo_K}
}
\]
where $F_{\eo_K} ([x_0 : \ldots : x_N]) = [x^q_0 : \ldots : x^q_N]$ with $q = p^r , r = (\kappa : \F_p)$. Then $\tau_{\eo_K}$ is finite and $\tau^{-1}_K (V_K) \to V_K$ is finite \'etale, where $V_K = Y_K \cap T$. We now reduce to  $\kappa$ and define a morphism $i : \Yh_{\kappa} \to \Yh'_{\kappa}$ over $\kappa$ by the commutative diagram
\[
\xymatrix{
\Yh_\kappa \ar@/^1pc/[drr]^{\subset}  \ar[dr]^i \ar@/_1pc/[ddr]_{\Fr^{\nu}_{\kappa} }\\
 & \Yh'_{\kappa} \ar@{^{(}->}[r] \ar[d]_{\tau_{\kappa}} & \Pa^N_{\kappa} \ar[d]^{\Fr^{\nu}_{\kappa}} \\
 & \Yh_{\kappa} \ar@{^{(}->}[r]^{} & \Pa^N_{\kappa}
 }
\]
where $F_{\eo_K} \otimes \kappa = \Fr_{\kappa} = \Fr^r_p$. Since $\Yh_{\kappa}$ is reduced, it follows from \cite{dewe1} Lemma 19 that $i$ induces an isomorphism $i: \Yh_{\kappa} \iso \Yh^{'\red}_{\kappa}$. Base changing to $\oZ_p$ and setting $\Yh' = \Yh'_{\eo_K} \otimes \oZ_p$ etc. we get proper morphisms
\[
\Yh' \xrightarrow{\tau} \Yh \xrightarrow{\pi} \eX
\]
such that the pullback of $\Eh_k$ along the following composition is a trivial bundle
\begin{equation}
\label{eq:cd11}
\Yh_k \overset{i_k}{\iso} \Yh^{'\red}_k \hookrightarrow \Yh'_k \xrightarrow{\tau_k} \Yh_k \xrightarrow{\pi_k} \eX_k \; .
\end{equation}
Note here that $\Yh^{'\red}_k = (\Yh'_{\kappa} \otimes_{\kappa} k)^{\red} = \Yh^{'\red}_{\kappa} \otimes_{\kappa} k$ since $\kappa$ is a perfect field and that $\tau_k \verk i_k = (\Fr_{\kappa} \otimes \id_k)^{\nu} = F^{\nu}$. By construction the open subscheme $V = V_K \otimes \oQ_p$ of the smooth variety $Y$ contains $\pi^{-1} (x)$ viewed as a subset of the closed points of $V$. Since $\pi$ is finite and hence closed, it follows that $U = X \ohne \pi (Y \ohne V)$ is an open neighborhood of $x$ in $X$ such that $\tV = \pi^{-1} (U) \subset V$. Since $\tau^{-1} (V) \to V$ is finite \'etale and surjective, it follows that the morphism $\tau^{-1} (\tV) \to \tV$ is finite \'etale and surjective as well. Let $V'_0 \neq \emptyset$ be a connected component of $\tau^{-1} (\tV)$. Then the open subscheme $V'_0$ of $Y'$ is a smooth quasi-projective variety over $\oQ_p$. Applying Lemma \ref{cd-t8} to the morphism $\Yh' \to \spec \oZ_p$, the open subscheme $\tau^{-1} (\tV)$ of $\Yh'_{\oQ_p}$ and the connected component $V'_0$ of $\tau^{-1} (\tV)$ we obtain a proper pair $(\Yh'_0 , V'_0)$ and a closed finitely presented immersion $\varepsilon : \Yh'_0 \hookrightarrow \Yh'$ such that $\varepsilon^{-1} (\tau^{-1} (\tV)) = \tau^{-1} (\tV) \cap Y'_0 = V'_0$. In particular the restriction of $\varepsilon$ to a morphism
\[
\varepsilon^{-1} (\tau^{-1} (\tV)) \longrightarrow \tau^{-1} (\tV)
\]
is finite \'etale as the inclusion of a connected component. 

Next we apply Lemma \ref{cd-t13} to the proper pair $(\Yh'_0 , V'_0)$. We obtain a very good model $\Zh$ over $\oZ_p$ together with a proper morphism $\lambda : \Zh \to \Yh'_0$ such that $\lambda^{-1} (V'_0) \stackrel{\sim}\rightarrow V'_0$ is an isomorphism. Consider the composition 
\[
\mu : \Zh \xrightarrow{\lambda} \Yh'_0 \overset{\varepsilon}{\hookrightarrow} \Yh' \xrightarrow{\tau} \Yh \xrightarrow{\pi} \eX \; .
\]
By our previous considerations $\mu^{-1} (U) \neq \emptyset$ and the restriction of $\mu$ to a morphism $\mu^{-1} (U) \to U$ is finite \'etale, hence surjective, since $U$ is connected. Moreover, since $\Zh_k$ is reduced, the composition
\[
\Zh_k \xrightarrow{(\varepsilon \verk \lambda)_k} \Yh'_k
\]
factors over $\Yh^{'\red}_k$. Since the pull-back of $\Eh_k$ along the composition \eqref{eq:cd11} is trivial it follows that $\mu^*_k \Eh_k$ is a trivial bundle on $\Zh_k$. 
\end{proof}

It remains to prove Lemma \ref{cd-t39}. For this we show the following fact: 

\begin{lemma} \label{cd-t40}
Let $K$ be a finite extension of $\Q_p$ with ring of integers $\eo_K$. Consider the open subscheme $T = \Ge^N_{m,K}$ of points with non-zero coordinates in $\Pa^N_K$. Given finitely many points $y_1 , \ldots , y_r \in \Pa^N (\oK)$ there is an element $g \in \PGL_{N+1} (\eo_K)$ such that $g (y_1) , \ldots , g (y_r) \in T (\oK)$.
\end{lemma}
\begin{proof}
Lift $y_{\alpha}$ to a vector $v_{\alpha} = (y_{\alpha , 0} , \ldots , y_{\alpha , N}) \in \oK^{N+1} \ohne \{ 0 \}$ for all $1 \le \alpha \le r$. 
Denote by $M_{N+1}(K) $ the algebra of all $(N+1) \times (N+1)$-matrices with entries in $K$. 
For $0 \le i \le N$ and $1 \le \alpha \le r$ consider the $K$-linear maps
\[
l_{\alpha , i} : M_{N+1} (K) \longrightarrow \oK \; , \; l_{\alpha,i} (A) = i\text{-th component of the vector} \; A v_{\alpha} \; .
\]
Then $\ker l_{\alpha , i}$ is at least $1$-codimensional in the $K$-vector space $M_{N+1} (K)$ and hence
\[
U = M_{N+1} (K) \ohne \bigcup_{i, \alpha} \ker l_{\alpha , i}
\]
is dense in $M_{N+1} (K)$ equipped with the topology induced by the valuation topology of $K$. It follows that $U$ contains an element $A$ of the open subset $\GL_{N+1} (\eo_K)$ of $M_{N+1} (K)$. For the points $y_{\alpha} = [y_{\alpha , 0} : \ldots : y_{\alpha , N}]$ of $\Pa^N (\oK)$ and the element $g := A \, \mod \eo^{\times}_K$ of $\PGL_{N+1} (\eo_K)$ we have $g (y_{\alpha}) \in T (\oK)$ for $1 \le \alpha \le r$. 
\end{proof}

\begin{proofof} {\bf Lemma \ref{cd-t39}}
Choose a closed immersion $\Wh \subset \Pa^N_{\eo_K}$ over $\eo_K$ and view $y_1 , \ldots , y_r$ as points of $\Pa^N (\oK)$. For $g$ as in Lemma \ref{cd-t40} the embedding $ \Wh \subset \Pa^N_{\eo_K} \xrightarrow{g} \Pa^N_{\eo_K}$ satisfies our claim.
\end{proofof}

\section{ $P$-divisible pullbacks on cohomology} \label{sec:cd6}

Let $K$ be a complete discretely valued field  with discrete valuation ring $\eo_K$ of mixed characteristic and with perfect residue field $k$. 
% and residue field $k$.
%, and let $S$ be the spectrum of $\eo_K$. 
For any $\eo_K$-scheme $\eX$ we denote the generic fiber by $\eX_K$. 

Let $\eX$ be a  reduced scheme which is proper and flat over $\eo_K$, and let $f: \Yh \rightarrow \eX$ be a proper morphism. We say that $f$ is finite \'etale surjective over an open subset $U$ of $\eX$ if the restriction $f|_{f^{-1} U}: f^{-1} U \rightarrow U$ is finite, \'etale and surjective. For any finite subset $F$ of closed points in the generic fiber $\eX_K$ we say that $f$ is finite \'etale around $F$, if there exists a  smooth, quasi-projective open subscheme $U$ of $\eX_K$ such that $ F \subset U$ and such that $f$ is finite \'etale surjective over $U$.  

Our goal in this section is to make pullback maps on coherent cohomology $p$-divisible on a proper covering which is finite \'etale around a given finite subset of the generic fiber.

\begin{theorem} \label{bhatt} Let $\eX$ be a reduced and proper scheme over $\eo_K$, and let $F \subset \eX$ a finite set of closed points which is contained in a connected, quasi-projective, smooth open subset of the generic fiber $\eX_K$. 
There is an $\eo_K$-scheme $\Zh$ and a proper $\eo_K$-morphism $f:\Zh \rightarrow \eX$ such that

i) For all $i \ge 1$, the pullback $f^\ast: H^i(\eX, {\cal O}_\eX) \rightarrow H^i (\Zh, {\cal O}_\Zh)$ is divisible by $p$ as a group homomorphism.

ii) The map  $f$ is finite \'etale surjective around $F$. 

\end{theorem}

 Note that this  result is a strengthening of \cite{bh1}, Theorem 1.2 for schemes over $\Spec \, {\cal{O}}_K$. The new aspect here is condition ii). We proceed as in the proof of \cite{bh1}, Proposition 2.9 paying close attention to the nature of all alterations involved in Bhatt's proof. For this purpose the notes \cite{bh2} will be a crucial tool. We begin by showing a preliminary result which we hope is of independent interest. It relies on de Jong's arguments in \cite{dej1}, section 5. 
 
 \begin{theorem}\label{thm:divisible} Let $R$ be a discrete valuation ring with quotient field $Q$, and put $S = \Spec R$. 
 Let $\Th$ be an integral  scheme together with a morphism $\Th \rightarrow S$ of finite type, and let $\eX$ be a stable, $m$-pointed curve over $\Th$ in the sense of \cite{kn}, Definition 1.1.
 Assume that there exists a non-empty normal open subscheme $U$ of $\Th$ contained in the generic fiber $\Th_Q$ such that $\eX_U = \eX \times_\Th U$ is smooth over $U$. Moreover, let $\varphi_U: \Yh_U \rightarrow \eX_U$ be a finite \'etale $U$-morphism. 
 
 Then there exists an integral,  normal scheme $\Th^\ast$ together with a proper $S$-morphism $\alpha: \Th^\ast \rightarrow \Th$ and a stable $n$-pointed curve $\Yh^\ast$ over $\Th^\ast$ for some $n \geq m$ such that the following properties hold for $U^\ast = \alpha^{-1} (U) \subset \Th^\ast$: 
 
 i) The map  $\alpha|_{U^\ast}: U^\ast \rightarrow U$ is finite, \'etale and surjective.
 
 ii) $\Yh^\ast_{U^\ast} = \Yh^\ast \times_{\Th^\ast} U^\ast$ is $U^\ast $-isomorphic to  the base change of $\Yh_U$ with $U^\ast \rightarrow U$.
 
 iii) There exists a morphism of $\Th^\ast$-curves 
 \[\varphi: \Yh^\ast \rightarrow \eX_{\Th^\ast} = \eX \times _\Th \Th^\ast \]
 extending the given morphism $\varphi_U \times_U U^\ast: \Yh^\ast_{U^\ast}  \simeq \Yh_U \times_U U^\ast \rightarrow \eX_U \times_U U^\ast$.

 \end{theorem}
 
 \begin{proof} Let $\sigma_1, \ldots, \sigma_m: U \rightarrow \eX_U$ denote the given sections restricted to $U$.
 After replacing $\Th$ by a finite cover $\Th_1 \rightarrow \Th$ which is finite  \'etale surjective over $U$ and  $\eX$ and $\Yh$ by their  base changes with $\Th_1$, we may assume that $\Yh_U$ has $n$ pairwise disjoint sections $\tau_1, \ldots, \tau_n$ over $U$ for some $n \geq m$ such that $\beta_U: \Yh_U \rightarrow \eX_U$ maps sections to sections in such a way that for all points $u$ in $U$ the preimage $\beta_U^{-1}(\sigma_i{u})$ consists only of points marked by the sections $\tau_1, \ldots, \tau_n$.

By Nagata compactification, we find a proper $\eX$-scheme $\Yh$ containing $\Yh_U$ as an open subscheme. Hence  we have a commutative diagram
\[
\xymatrix{
  \Yh_U \ar[d] \ar@{^{(}->}[r] & \Yh \ar[d]\\
\eX_U \ar@{^{(}->}[r] & \eX.
}
\]
Let $g$ be the genus of the fibers of $\Yh_U/U$. 
Now we argue as in \cite{dej1}, 5.13.  Let $\overline{\mathcal{M}}_{g,n}$ be the algebraic stack over $\mathbb{Z}$ classifying stable $n$-pointed curves of genus $g$, and let $\mathcal{M}_{g,n}$ be the open substack classifying smooth $n$-pointed curves. The smooth  $n$-pointed curve $\Yh_U$  gives rise to a $1$-morphism $U \rightarrow {\mathcal{M}}_{g,n}$. We choose a prime $l $ invertible in $R$, and consider the finite \'etale cover $M \rightarrow \mathcal{M}_{g,n}[1/l]$ of \cite{dej1}, 2.24. Here $M$ is a scheme over $\Z[1/l]$. This cover can be extended to a morphism $\overline{M} \rightarrow \overline{\mathcal{M}}_{g,n}[1/l]$ over $\Z[1/l]$, where $\overline{M}$ is a projective scheme over $\Z[1/l]$.
As usual, we denote the base changes of $M$ and $\overline{M}$ with $S$ by  $M_S$ and $\overline{M}_S$. 

The fiber product ${M}_S \times_{{\mathcal{M}}_{g,n,S}} U $ (which is a scheme) is finite and \'etale over $U$. Let $U_0$ be a connected component dominating $U$. It is irreducible, normal and  finite \'etale surjective over $U$.  Now consider the canonical immersion ${M}_S \times_{{\mathcal{M}}_{g,n,S}} U \rightarrow \overline{M}_S \times_S \Th$, and let $\Th_0$ be the closure of the image of ${U}_0$ in $\overline{M}_S \times_S \Th$ with the reduced induced structure. Then $\alpha_0:  \Th_0 \rightarrow \Th$ is projective since $\overline{M}_S$ is projective, and $\alpha_0^{-1}(U) \simeq U_0$ is finite \'etale surjective over $U$. 

By construction, there exists a stable $n$-pointed curve $\Yh_0$ over ${\Th_0}$ and an isomorphism of stable $n$-pointed curves $\Yh_0 \times_{\Th_0} U_0  \rightarrow \Yh_U \times_U U_0$. 
Now we argue as in \cite{dej2}, 3.9-3.13 to define a proper map $\Th^\ast \rightarrow \Th_0$ which is an isomorphism  on the preimage $U^\ast$ of $U_0$ such that, putting $\Yh^\ast = \Yh_0 \times_{\Th_0} \Th^\ast$, the morphism $\Yh^\ast \times_{\Th^\ast} U^\ast \simeq \eX_U \times_U  U^\ast$ extends to a morphism $\Yh^\ast \rightarrow  \eX \times_\Th \Th^\ast$.  Properness is shown by the valuative criterion which amounts to checking that a similar extension result holds for stable curves over discrete valuation ring. This follows as in the proof of   \cite{liulo}, Proposition 4.4.

By construction, $\Th^\ast$ is integral, and by replacing it with its normalization, we may assume that it is normal. If we define $\alpha: \Th^\ast \rightarrow \Th_0 \rightarrow \Th$ as the composition of the two maps constructed previously, our claim follows. 
\end{proof}

Now we can show the result stated at the beginning of this section.

 \begin{proof}{ \bf of Theorem \ref{bhatt}}
 We may replace $\eX$ by the closure of $U$ in $\eX$ with its reduced structure, and thus assume that $\eX$ is integral. Then $\eX$ is flat over $S = \Spec \eo_K$. 
 We proceed as in \cite{bh1}, Proposition 2.10 by induction over the dimension of the generic fiber $\eX_K$. 
 If $\eX_K$ is a curve, the result follows as in the proof of  \cite{dewe1}, Theorem 11.
Therefore we may assume that $\eX_K$ has dimension $d>1$. 
Note that in order to prove our claim, we may replace $\eX$ by any reduced scheme $\eX'$ such that there exists a proper $S$-morphism $h: \eX' \rightarrow \eX$ which is finite \'etale surjective over an open subset $U$ of the generic fiber containing $F$. 
Moreover, after replacing $\eX'$ by the closure of an irreducible component of $h^{-1}(U)$, we may additionally take any such $\eX'$ to be irreducible. 

By \cite{bh2}, 5.6 we may therefore assume that 
$\eX \rightarrow S$ factors through a projective map $\phi: \eX \rightarrow \Th$ with the following properties: 
$\Th$ is integral, proper and flat over $S$, and there exists a smooth  quasi-projective open $V \subset \eX_K$ containing $F$ and a smooth quasi-projective open $U \subset \Th_K$ such that $V$ is mapped to $U$ via $\phi$ and such that 
$(V, \eX) \rightarrow (U,\Th)$ is a stable $m$-pointed compactified elementary fibration 
in the sense of \cite{bh2}, Definition 5.1.
This implies that $\phi_{|U}: \eX_U \rightarrow U$ is smooth and that all fibers of $\phi$ are geometrically connected of equidimension one with dense smooth locus. Moreover, there are $m$ sections of $\phi: \eX \rightarrow \Th$ for some $m \geq 0$ such that $\phi$ is a relative stable $m$-pointed curve.

Following roughly the notation of \cite{bh1}, Proposition 2.10, we denote by $f: \eX \rightarrow S$ and $f': \Th \rightarrow S$ the structure morphisms. By induction, we know that our claim holds for $\Th$. Hence there exists a proper $S$-morphism $\pi': \Th' \rightarrow \Th$ which is finite \'etale surjective around $\phi(F)$ such that, with $g' = f' \circ \pi'$, 
 the natural homomorphism  $\pi^{' \ast}: R^i f'_\ast \Oh_\Th \rightarrow R^i g'_\ast \Oh_{\Th'}$ is divisible by $p$ (as a homomorphism of sheaves of abelian groups). After shrinking $U$ if necessary, we may assume that $\pi'$ is finite \'etale surjective over $U$. As explained above, we may further assume that $\Th'$ is integral. 
 The base change $\phi': \eX' = \eX \times_\Th \Th' \rightarrow \Th'$ is a stable $n$-pointed curve. Put $U' = \pi^{' -1} (U)$. Then $\eX'_{U'} = \eX' \times_{\Th'} U'$ is smooth and projective over $U'$. 

The functor $\mbox{Pic}^0_{\eX'/\Th'}$ is represented by a semi-abelian $\Th'$-scheme, and  $\mbox{Pic}^0_{\eX'_{U'}/U'}$ is an abelian scheme, see \cite{blr},  section 9.4. We denote the  dual abelian scheme by $\Alb_{\eX'_{U'}/U'}$. Fix a section $x \in \eX'_{U'}(U')$. By the universal property of the Picard scheme there exists a canonical $U'$-morphism 
\[\eX'_{U'} \rightarrow \Alb_{\eX'_{U'}/U'}\]
sending $x$ to the zero section. We  define $\Yh_{U'}$ as the base change with respect to the $p$-multiplication on the abelian scheme $\Alb_{\eX'_{U'}/U'}$:
\[
\xymatrix{
  \Yh_{U' }\ar[d]^{q} \ar[r] & \Alb_{\eX'_{U'}/U'}  \ar[d]^p \\
\eX'_{U'} \ar[r] & \Alb_{\eX'_{U'}/U'}
}
\]
The map $q$ is finite and \'etale, since $U'$ is contained in the generic fiber $\Th'_K$ which has characteristic zero. 
Note that the geometric fibers of $\Yh_{U'}$ are connected as pullbacks of a curve under the $p$-multiplication map on the Jacobian. 

By Theorem \ref{thm:divisible} we find an integral, normal $\Th''$ together with  a proper $S$-morphism $\pi'': \Th'' \rightarrow \Th'$ which is finite \'etale surjective over $U'$ and a stable $m$-pointed curve $\phi'': \eX'' \rightarrow \Th''$ together with a morphism $\eX'' \rightarrow  \eX'  $ sitting in the commutativie diagram
\[
\xymatrix{
  \eX'' \ar[d]^{\phi''} \ar[r] & \eX'  \ar[d]^{\phi'}\\
\Th'' \ar[r]^{\pi''}& \Th'
}
\]
 such that we have the following properties for  $U'' = (\pi'')^{-1} (U') \subset \Th''$:
The restriction of $\pi''$ to $U''$ is finite \'etale surjective,  the scheme $\eX'' \times_{\Th'' } U'' $ is isomorphic to $\Yh_{U' }\times_{U'} U''$, and  the induced morphism $a: \eX'' \rightarrow \eX' \times_{\Th'} \Th'' = \eX'_{\Th''}$  extends the base change of $q$ to $U''$. 
Note that there is a natural identification of $R^1 \phi'_\ast \Oh_{\eX'}$ (or $R^1 \phi''_\ast \Oh_{\eX''}$, respectively) with the Lie algebra of the semiabelian $\Th'$- scheme $\mathrm{Pic}^0(\eX'/\Th')$ (or the $\Th''$-scheme $\mathrm{Pic}^0(\eX''/\Th'')$, respectively), see \cite{blr}, 8.4, Theorem 1. The map on Lie algebras induced by 
\[\mathrm{Pic}^0 (a): \mathrm{Pic}^0(\eX'_{\Th''} / \Th'' )  \rightarrow \mathrm{Pic}^0(\eX''/  \Th'').\]
 is divisible by $p$ after restriction to $U''$, hence everywhere by \cite{fach}, Proposition I.2.7. Hence the natural homomorphism
  \[(\pi'')^\ast R^1 \phi'_\ast \Oh_{\eX'} \rightarrow R^1 \phi''_\ast \Oh_{\eX''}\]
   is divisible by $p$ as a group homomorphism on $\Th''$. 
As in the proof of \cite{bh1}, Proposition 2.10, we can now argue with a splitting of the Leray spectral sequence to deduce that  the map $ \eX'' \rightarrow \eX' \rightarrow \eX$ 
satisfies part i) of our claim. Since it is finite and \'etale over $\eX_U$ we also proved part ii). 
\end{proof}
 
In the next section we will need the following variant of Theorem \ref{bhatt}.

\begin{cor}
\label{cd-t44}
Let $(\Yh , W)$ be a proper pair over $\oZ_p$ and let $\emptyset \neq F \subset W$ be a finite set of closed points. Then there are a very good model $\Yh'$ over $\oZ_p$ and a proper finitely presented $\oZ_p$-morphism $f : \Yh' \to \Yh$ such that \\
i) For all $i \ge 1$, the pullback $f^* : H^i (\Yh , \Oh) \to H^i (\Yh' , \Oh)$ is divisible by $p$ as a group homomorphism.\\
ii) $f^{-1} (V) \to V$ is finite \'etale surjective  for some open set $F \subset V \subset W$.
\end{cor}

\begin{proof}
Using Lemma \ref{cd-t13} we may assume that $\Yh$ is a very good model. After descending $\Yh$ to $\eo_K$ for some finite extension $K$ of $\Q_p$ we can apply Theorem \ref{bhatt}. Base extending back to $\oZ_p$ we have obtained a proper morphism $f : \Yh' \to \Yh$ with properties i) and ii). Using Lemma \ref{cd-t8} we may assume that $(\Yh' , f^{-1} (V))$ is a proper pair. We can now apply Lemma \ref{cd-t13} to improve $\Yh'$ to a very good model.
\end{proof}

\section{Lifting trivializing covers} \label{sec:cd5}
Let $X$ be a smooth complete variety over $\oQ_p$ and $\eX$ a model of $X$ over $\oZ_p$. By Lemma \ref{cd-t7} the special fibre $\eX_k$ of $\eX$ is connected and $\eX$ is integral with $X$ open and dense in $\eX$. In Theorem \ref{cd-t22} we have shown that every vector bundle with numerically flat reduction lies in $\Bh_{\emm} (\eX_{\eo})$ for the maximal ideal $\emm$ in $\oZ_p$, see Definition  \ref{cd-t21}.  In this section, we prove the following important property for these bundles.

\begin{theorem}\label{cd-t23}
For every $\omega \in \oZ_p$ with $0 < |\omega| < 1$ we have
$\Bh_\emm (\eX_{\eo}) = \Bh_{\omega} (\eX_{\eo})$.
\end{theorem}

\begin{proof}
Consider $\omega \in \oZ_p$ with $0 < |\omega| < 1$. Then $k = \oZ_p / \emm $ is a quotient of $\oZ_p / \omega \oZ_p$ and it follows from Definition \ref{cd-t21} that $\Bh_{\omega} (\eX_{\eo}) \subset \Bh_\emm (\eX_{\eo})$. 

Thus let $\Eh$ be a vector bundle in $\Bh_\emm(\eX_{\eo})$. For every point $x \in X (\oQ_p)$ we then have a very good model $\Yh$ over $\oZ_p$ together with a proper morphism $\pi : \Yh \to \eX$ such that $\emptyset \neq \pi^{-1} (U) \to U$ is finite \'etale for an open neighborhood $U$ of $x$ and such that $\pi^*_k \Eh_k$ is a trivial bundle on $\eX_k$. The family $(\eX , \Eh \otimes  \eo / p , \pi : \Yh \to \eX)$ where $\Eh \otimes \eo / p = \Eh \otimes \oZ_p / p$ is a vector bundle over $\eX_{\eo} \otimes \eo / p = \eX \otimes \oZ_p / p$ descends to a family $(\eX_{\eo_K} , \Fh , \pi_{\eo_K} : \Yh_{\eo_K} \to \eX_{\eo_K})$ over a finite extension $K$ of $\Q_p$. Here $\eX_{\eo_K} \otimes_{\eo_K} \oZ_p = \eX , \pi_{\eo_K} \otimes_{\eo_K} \oZ_p = \pi$ and $\Fh$ is a vector bundle on $\eX_{\eo_K} \otimes \eo_K / p$ with $\Fh \otimes_{\eo_K / p } \oZ_p / p = \Eh \otimes \eo / p$ whose restriction to the special fibre $\eX_{\eo_K} \otimes \eo_K / \ep$ is trivialised by pullback along $\pi_{\eo_K} \otimes \eo_K / \ep$. Here $\ep$ denotes the valuation ideal of $\eo_K$. Set
\[
\eo_{\nu} = \eo / \ep^{\nu} \eo = \oZ_p / \ep^{\nu} \oZ_p \; .
\]
Note that this notation differs from the one in section \ref{sec:cd3} if $K / \Q_p$ is ramified. Let $\eX_\nu, \Yh_\nu, \pi_{\nu} , \Eh_{\nu}$ etc. be the base changes of $\eX$, $\Yh$, $\pi$ and $\Eh$ etc. with $\eo_{\nu}$ over $\oZ_p$ and $\eo$, respectively. Since $\pi_1$ is also the base change of $\pi_{\eo_K} \otimes \eo_K / \ep$ with $\eo_1$ it follows that $\pi^*_1 (\Eh_1)$ is a trivial bundle on $\Yh_1$. Let $t$ be a prime element of $\eo_K$ so that $\ep = t \eo_K$. We have shown that $\Eh$ is in $\Bh_t (\eX_{\eo})$. The next claim implies that $\Eh$ is in $\Bh_{t^{\nu}} (\eX_{\eo})$ for all $\nu \ge 1$ and hence in $\Bh_{\omega} (\eX_{\eo})$ since $\Bh_{t^{\nu}} (\eX_{\eo}) \subset \Bh_{\omega} (\eX_{\eo}) $ if $|t^{\nu}| \le |\omega|$, which is the case for $\nu$ big enough. It thus remains to show the following assertion:

\begin{claim}
Given $\nu \ge 2$ consider a point $x \in X (\oQ_p)$ and a very good model $\Yh$ over $\oZ_p$ with a proper morphism $\pi : \Yh \to \eX$ such that $\pi^{-1}(U) \to U$ is finite \'etale surjective for an open neighborhood $U \subset X$ of $x$ and such that $\pi^*_{\nu -1} (\Eh_{\nu-1})$ is a trivial bundle on $\Yh_{\nu-1}$. Then there is a very good model $\Zh$ over $\oZ_p$ with a proper morphism $\mu : \Zh \to \eX$ such that $\mu^{-1} (\tU) \to \tU$ is finite \'etale surjective for some open neighborhood $\tU \subset U$ of $x$  and such that $\mu^*_{\nu} (\Eh_{\nu})$ is a trivial bundle on $\Zh_{\nu}$.
\end{claim}

In order to prove the claim we first identify the obstruction to $\pi^*_{\nu} (\Eh_{\nu})$ being a trivial bundle. Consider the ideal $J = t^{\nu-1} \Oh_{\Yh_{\nu}}$ in $\Oh_{\Yh_{\nu}}$ corresponding to the closed subscheme $\Yh_{\nu-1} \subset \Yh_{\nu}$. Consider the short exact sequence of Zariski sheaves of groups where $r = \rk \Eh$ and $f (A) = 1+A$
\[
0 \longrightarrow M_r (J) \xrightarrow{f} \GL_r (\Oh_{\Yh_{\nu}}) \longrightarrow \GL_r (\Oh_{\Yh_{\nu-1}}) \longrightarrow 1 \; .
\]
Note that $f$ is a homomorphism of groups because $J^2 = 0$. Right exactness follows from formal smoothness of $\GL_r$ over $\Z$. Thus we have an exact sequence of pointed sets
\begin{equation}
\label{eq:cd12}
H^1 (\Yh_{\nu} , M_r (J)) \xrightarrow{f} H^1 (\Yh_{\nu} , \GL_r (\Oh)) \longrightarrow H^1 (\Yh_{\nu-1} , \GL_r (\Oh)) \; .
\end{equation}
The class of $\pi^*_{\nu} (\Eh_{\nu})$ in the middle is mapped to the class of $\pi^*_{\nu-1} (\Eh_{\nu-1})$  on the right i.e. to the trivial class. Hence we have
\[
[\pi^*_{\nu} (\Eh_{\nu})] = f (a)
\]
for some class $a = (a_{kl})$ with $a_{kl} \in H^1 (\Yh_{\nu} , J)$. This also follows using \cite{Gir} VII, Th\'eor\`eme 1.3.1. Consider the closed immersion $i : \Yh_1 \hookrightarrow \Yh_{\nu}$. Multiplication by $t^{\nu-1}$ provides an isomorphism of sheaves of $\Oh_{\Yh_{\nu}} / t$-modules since $\Yh$ is integral
\[
t^{\nu-1} : i_* \Oh_{\Yh_1} \iso J = t^{\nu-1} \Oh_{\Yh_{\nu}} \; .
\]
This gives an isomorphism of $\oZ_p / t$-modules
\[
t^{\nu-1} : H^1 (\Yh_1 , \Oh) \iso H^1 (\Yh_{\nu} , J) \; .
\]
Let $b_{kl} \in H^1 (\Yh_1 , \Oh)$ denote the classes corresponding to the obstructions $a_{kl}$. Consider the finite subset $\pi^{-1} (x) \subset Y (\oQ_p)$. Applying Corollary \ref{cd-t44} to $F = \pi^{-1} (x)$ and $\Yh$ we obtain a very good model $\Yh'$ over $\oZ_p$ together with a proper, finitely presented  $\oZ_p$-morphism $f : \Yh' \to \Yh$ such that for each $i \ge 1$ the pullback
\[
f^* : H^i (\Yh , \Oh) \longrightarrow H^i (\Yh' , \Oh)
\]
is divisible by $p$ as a $\oZ_p$-linear map. Moreover $f$ is finite \'etale surjective around $\pi^{-1} (x)$ i.e. there is an open neighborhood $\pi^{-1} (x) \subset V \subset \pi^{-1} (U)$ such that $f^{-1} (V) \to V$ is finite \'etale surjective. Now $U_1 = X \ohne \pi (Y \ohne V) \subset U$ is an open neighborhood of $x$ in $X$ such that $\pi^{-1} (U_1) \subset V$. Thus $f^{-1} (\pi^{-1} (U_1)) \to \pi^{-1} (U_1)$ and therefore $(\pi \verk f)^{-1} (U_1) \to U_1$ is finite \'etale surjective. Consider the short exact sequence (note that $\Yh$ is integral)
\[
0 \longrightarrow \Oh_{\Yh} \xrightarrow{t} \Oh_{\Yh} \longrightarrow \Oh_{\Yh_1} \longrightarrow 0
\]
and the corresponding one on $\Yh'$. We get a commutative diagram
\[
\xymatrix{
0 \ar[r] & H^1 (\Yh , \Oh) / t \ar[r] \ar[d]^{f^*} & H^1 (\Yh_1 , \Oh) \ar[r]^{\delta} \ar[d]^{f^*} & H^2 (\Yh , \Oh)_t \ar[r] \ar[d]^{f^*} & 0 \\
0 \ar[r] & H^1 (\Yh' , \Oh) / t \ar[r] & H^1 (\Yh'_1 , \Oh) \ar[r]^{\delta} & H^2 (\Yh' , \Oh)_t \ar[r] & 0 \, .
}
\]
Here we use the notation $H / t = H / tH$ and $H_t = \Ker (t : H \to H)$. Since the linear map $f^*$ is divisible by $p$ on $H^2$ and $t$ divides $p$, the induced map on the right is zero:
\[
0 = f^* : H^2 (\Yh , \Oh)_t \longrightarrow H^2 (\Yh' , \Oh)_t \; .
\]
%The map on the left is also zero but this does not help at the moment. All we get is that 
Hence for all $k,l$ we find
\begin{equation}
 \label{eq:cd13}
 f^* (b_{kl}) \in \Imm (H^1 (\Yh' , \Oh) / t \longrightarrow H^1 (\Yh'_1 , \Oh)) \; .
\end{equation}
Now we use the previous argument a second time.  By Corollary \ref{cd-t44} applied to $\Yh'$ and the finite set $(\pi \verk f)^{-1} (x)$ we obtain a very good model $\Zh$ over $\oZ_p$ together with a proper morphism $g : \Zh \to \Yh'$ which is finite \'etale surjective around $(\pi \verk f)^{-1} (x)$ such that the map
\[
g^* : H^1 (\Yh' , \Oh) \longrightarrow H^1 (\Zh , \Oh)
\]
is divisible by $p$ as a $\oZ_p$-linear map. This time we look at the map on the left in the resulting diagram, which implies vanishing of the left vertical map in the  commutative diagram
\[
\xymatrix{
H^1 (\Yh' , \Oh) / t \ar[r] \ar[d]^{g^* = 0} & H^1 (\Yh'_1 , \Oh) \ar[d]^{g^*} \\
H^1 (\Zh , \Oh) / t \ar[r] & H^1 (\Zh_1 , \Oh).
}
\]
Hence it follows from \eqref{eq:cd13} that
\[
g^* (f^* (b_{kl}) = 0 \quad \text{in} \; H^1 (\Zh_1 , \Oh) \; .
\]
Hence we have $(f \verk g)^* (a_{kl}) = 0$ in $H^1 (\Zh_{\nu} , J)$ where $J = t^{\nu-1} \Oh_{\Zh_{\nu}}$. The exact sequence \eqref{eq:cd12} for $\Zh$ instead of $\Yh$ now shows that the class $[(\pi \verk f \verk g)^*_{\nu} (\Eh_{\nu})]$ is trivial in $H^1 (\Zh_{\nu} , \GL_r (\Oh))$. Thus $\mu = \pi \verk f \verk g : \Zh \to \Yh' \to \Yh \to \eX$ is a proper morphism from a very good model $\Zh$ over $\oZ_p$ to $\eX$ such that $\mu^*_{\nu} (\Eh_{\nu})$ is a trivial bundle. Since $g : \Zh \to \Yh'$ is finite \'etale surjective around $(\pi \verk f)^{-1} (x)$ we find as above an open neighborhood $\tU \subset U_1$ of $x$ in $X$ such that $\mu^{-1} (\tU) \to \tU$ is finite \'etale surjective. Thus the claim is proved.
\end{proof}

\section{Parallel transport over the ring of integers} \label{sec:cd31}

We can now construct parallel transport for all bundles with numerically flat reduction.

Let $X$ be a smooth, complete variety over $\oQ_p$ and $\eX$ a model of $X$ over $\oZ_p$. Let $\Eh$ be a rank $r$ vector bundle on $\eX_{\eo}$ with numerically flat reduction, i.e. $\Eh \in \Bh^s (\eX_{\eo})$. Fix an integer $n \ge 1$. By Theorems \ref{cd-t22} and \ref{cd-t23} with $\omega = p^n$ there exists an open covering $\eU = (U)$ of $X$ together with good trivializing covers $\pi : \Yh \to \eX$ for $\Eh_n$ over $(\eX , U)$. We may assume that $\eU$ is stable under finite intersections. Then Construction \ref{cd-t26} gives functors
\[
\rho_{\Eh,n} (U) : \Pi_1 (U) \longrightarrow \Free_r (\eo_n)
\]
for all $U \in \eU$ which depend only on $(\eX , U)$ and $\Eh$. Moreover, for any inclusion $j : V \hookrightarrow U$ of open sets in $\eU$, formula \eqref{eq:cd5} of Lemma \ref{cd-t27} applied to $\alpha = \id_{\eX}$ shows that
\[
\rho_{\Eh ,n} (U) \verk j_* = \rho_{\Eh,n} (V) : \Pi_1 (V) \longrightarrow \Free_r (\eo_n) \; .
\]

\begin{defn}\label{cd-t28} In the situation above, we define $\rho_{\Eh,n}$ as  the unique functor 
\begin{equation} \label{eq:cd7}
\rho_{\Eh,n} : \Pi_1 (X) \longrightarrow \Free_r (\eo_n)
\end{equation}
such that
\begin{equation} \label{eq:cd7a}
\rho_{\Eh,n} (U) = \rho_{\Eh,n} \verk j_{U^*} \quad \text{for all} \; U \in \eU \; .
\end{equation}
Here $j_U : U \hookrightarrow X$ is the open immersion. 
\end{defn}

Note that since $\Free_r (\eo_n)$ is a quasi-finite groupoid, it follows from Theorem \ref{cd-t24} that this is well-defined.

Since any two open coverings of $X$ can be refined by (an intersection stable) open covering it follows formally that $\rho_{\Eh ,n}$ depends only on $\eX$ and $\Eh$ and not on the choice of $\eU$. Note that by definition: 
\begin{equation}\label{eq:cd8}
\rho_{\Eh,n} (x) = \Eh_{x_n} \quad \text{for all} \; x \in X (\C_p) \; .
\end{equation}
In particular this parallel transport gives a representation
\begin{equation} \label{eq:cd9}
\rho_{\Eh,n,x} : \pi_1 (X,x) \longrightarrow \GL (\Eh_{x_n}) := \Aut_{\eo_n} (\Eh_{x_n})
\end{equation}
for any $x \in X (\C_p)$. 

We now discuss two functorial properties of $\rho_{\Eh,n}$. 

\begin{prop} \label{cd-t29}
Let $X$ and $\tX$ be smooth, complete varieties over $\oQ_p$ with models $\eX$ and $\teX$ over $\oZ_p$. For any morphism $\alpha : \teX \to \eX$ and any vector bundle $\Eh$ in $\Bh^s (\eX_{\eo})$ the bundle $\alpha^* \Eh$ is in $\Bh^s (\teX_{\eo})$ and we have
\[
\rho_{\alpha^* \Eh , n} = \rho_{\Eh,n} \verk \alpha_* : \Pi_1 (\tX) \longrightarrow \Free_r (\eo_n) \quad \text{if} \; r = \rank \, \Eh \; .
\]
\end{prop}

\begin{proof}
With $\Eh_k$ the special fibre $(\alpha^*_{\eo} \Eh)_k = \alpha^*_k \Eh_k$ of $\alpha^*_{\eo} \Eh$ is numerically flat as well. Hence $\alpha^* \Eh$ is in $\Bh^s (\teX_{\eo})$. Let $\eU = \{ U \}$ be an intersection stable open covering of $X$ for which good trivializing covers of $\Eh_n$ over $(\eX , U)$ exist for all $U \in \eU$. Consider the intersection stable open covering $\alpha^{-1} \eU = \{ \alpha^{-1} (U) \}$ of $\tX$. By Lemma \ref{cd-t27} there exist good trivializing covers of $(\alpha^* \Eh)_n$ over the proper pairs $(\teX , \alpha^{-1} (U))$ for all $U \in \eU$. For the functors $\rho_{\alpha^* \Eh , n} (\alpha^{-1} (U))$ obtained by Construction \ref{cd-t26} we have for all $U \in \eU$
\begin{align*}
\rho_{\alpha^* \Eh , n} (\alpha^{-1} (U)) & = \rho_{\Eh, n} (U) \verk (\alpha \, |_{\alpha^{-1} (U)})_* \quad \text{by Lemma \ref{cd-t27}, \eqref{eq:cd5}}\\
& = \rho_{\Eh , n} \verk j_{U^*} \verk (\alpha \, |_{\alpha^{-1} (U)})_* \quad \text{by \eqref{eq:cd7a}} \\
& = \rho_{\Eh , n} \verk \alpha_* \verk j_{\alpha^{-1} (U)^*} \; .
\end{align*}
Here $j_{\alpha^{-1} (U)} : \alpha^{-1} (U) \to \tX$ is the inclusion. By the uniqueness assertion in the Seifert-van Kampen Theorem \ref{cd-t24} it follows that
\[
\rho_{\alpha^* \Eh ,n} = \rho_{\Eh , n} \verk \alpha_* \; .
\]
Here we have used that $\Free_r (\eo_n)$ is a quasi-finite groupoid. 
\end{proof}

The next proposition investigates functoriality with respect to conjugation. Let $G_{\Q_p} = \Gal (\oQ_p / \Q_p)$ be the absolute Galois group of $\Q_p$. It is also the group of continuous automorphisms of $\C_p$ over $\Q_p$. For a scheme $\eX$ over $\oZ_p$ and a vector bundle $\Eh$ on $\eX_\eo$ we set $^{\sigma} \eX = \eX \otimes_{\oZ_p , \sigma} \oZ_p$ and $^{\sigma} \Eh = \Eh \otimes_{\eo , \sigma} \eo$, which is a vector bundle on $^{\sigma} (\eX_{\eo}) = ({}^{\sigma}\! \eX)_{\eo}$.

\begin{prop} \label{cd-t30}
For every $\sigma$ in $G_{\Q_p}$ and each vector bundle $\Eh$ in $\Bh^s (\eX_{\eo})$ the conjugate bundle ${}^{\sigma}\!\Eh$ is in $\Bh^s ({}^{\sigma}\! \eX_{\eo})$ and the following diagram commutes:
\[
\xymatrix{
\Pi_1 (\eX) \ar[d]^{\wr}_{\sigma_*} \ar[r]^-{\rho_{\Eh ,n} } & \Free_r (\eo_n) \ar[d]^{\wr\, \sigma_*} \\
\Pi_1 ({}^{\sigma}\! \eX) \ar[r]^-{\rho_{{}^{\sigma}\!\Eh ,n }} & \Free_r (\eo_n)
}
\]
Here the left vertical isomorphism of groupoids is induced by the isomorphism $\sigma : \eX \iso \,{}^{\sigma}\!\eX$ of schemes. The right vertical isomorphism sends a free $\eo_n$ module $\Gamma$ of rank $r$ to $\Gamma \otimes_{\eo_n, \sigma} \eo_n$ and a homomorphism $f : \Gamma_1 \to \Gamma_2$ of such modules to $f \otimes_{\eo_n, \sigma} \eo_n$. 
\end{prop}

The proof of the Proposition is straightforward.

Let $\Rep_{\Pi_1 (Z)} (\eo_n)$ be the $\eo_n$-linear exact category of functors from $\Pi_1 (Z)$ to $\Free (\eo_n)$, the category of free $\eo_n$-modules of finite rank. It is equipped with tensor products, duals, internal homs and exterior powers.

\begin{theorem} \label{cd-t30n}
Parallel transport induces an $\eo$-linear exact functor
\[
\rho_n : \Bh^s (\eX_{\eo}) \longrightarrow \Rep_{\Pi_1 (X)} (\eo_n)
\]
by mapping $\Eh$ to $\rho_{\Eh , n}$ and a morphism $f : \Eh \to \tEh$ over $\eX_{\eo}$ to the natural transformation $\rho_{f,n} : \rho_{\Eh, n} \to \rho_{\tEh, n}$ given by the maps $f_{x_n} = x^*_n f : \Eh_{x_n} \to \tEh_{x_n}$ for $x \in X (\C_p) = \Ob (\Pi_1 (X))$. The functor $\rho_n$ commutes with tensor products, duals, internal homs and exterior powers.
\end{theorem}

\begin{proof}
First we have to show that $\rho_{f,n}$ is natural, i.e. that the following diagram commutes for all paths $\gamma$ from a point $x \in X (\C_p)$ to a point $x' \in X (\C_p)$:
\begin{equation}
\label{eq:cd10}
\xymatrix{
\Eh_{x_n} \ar[d]_{\rho_{\Eh , n} (\gamma)} \ar[r]^{f_{x_n}} & \tEh_{x_n} \ar[d]^{\rho_{\tEh ,n} (\gamma)} \\
\Eh_{x'_n} \ar[r]^{f_{x'_n}} & \tEh_{x'_n} \; .
}
\end{equation}
This can be seen as follows. By Theorems \ref{cd-t22} and \ref{cd-t23} we have $\Bh^s (\eX_{\eo}) = \Bh_{p^n} (\eX_{\eo})$. Because of Lemma \ref{cd-t17} there are an open quasiprojective (smooth) subscheme $U \subset X$ with $x \in U (\C_p)$ and a very good cover $\pi : \Yh \to \eX$ over the proper pair $(\eX , U)$ which is trivializing for both $\Eh_n$ and $\tEh_n$. Similarly for $x'$. Because of Theorem \ref{cd-t24} i) and Proposition \ref{cd-t29} we may assume that both $x$ and $x'$ lie in $U$ when we check that \eqref{eq:cd10} commutes. Noting that both $\pi^*_n \Eh_n$ and $\pi^*_n \tEh_n$ are trivial bundles, this case follows from the following commutative diagram where $y \in Y (\C_p)$ is a chosen point over $x$
\[
\xymatrix{
\Eh_{x_n} \ar[rr]^{f_{x_n}} & & \tEh_{x_n} \\
\Gamma (\Yh_n , \pi^*_n \Eh) \ar[rr]^{\Gamma (\Yh_n , \pi^*_n f_n)} \ar[u]^{y^*_n}_{\wr} \ar[d]^{\wr}_{(\gamma y)^*_n} & & \Gamma (\Yh_n, \pi^*_n \tEh_n) \ar[u]_{\wr \, y^*_n} \ar[d]^{\wr \, (\gamma y)^*_n} \\
\Eh_{x'_n} \ar[rr]^{f_{x'_n}} & & \tEh_{x'_n} \; .
}
\]
Exactness of $\rho_n$ is clear. Compatibility with tensor products etc. follows from our construction, using Lemma \ref{cd-t17} in a similar way. 
\end{proof}

For a morphism $\alpha : \teX \to \eX$ of models over $\oZ_p$ it follows from Proposition \ref{cd-t29} that pullback by $\alpha$ induces a functor $\alpha^* : \Bh^s (\eX_{\eo}) \to \Bh^s (\teX_{\eo})$. The next result details how this  is related with our parallel transport  $\rho_n$. The generic fibre of $\alpha$ induces a functor
\[
A (\alpha) : \Rep_{\Pi_1 (X)} (\eo_n) \longrightarrow \Rep_{\Pi_1 (\tX)} (\eo_n)
\]
as follows. For an object $M$ of $\Rep_{\Pi_1 (X)} (\eo_n)$ we define $A (\alpha) M$ to be the composed functor
\[
A (\alpha) M : \Pi_1 (\tX) \xrightarrow{\alpha_*} \Pi_1 (X) \xrightarrow{M} \Free (\eo_n) \; .
\]
For a morphism $f : M \to N$ in $\Rep_{\Pi_1 (X)} (\eo_n)$ given by $\eo_n$-linear maps $f_x : M_x \to N_x$ for $x \in X (\C_p)$ we define $A (\alpha) f$ as the family of maps for $\tx \in \tX (\C_p)$
\[
(A (\alpha) M)_{\tx} = M_{\alpha (\tx)} \xrightarrow{f_{\alpha (\tx)}} N_{\alpha (\tx)} = (A (\alpha) N)_{\tx} \; .
\]

\begin{prop} \label{cd-t31}
Consider models $\eX , \teX$ over $\oZ_p$ of smooth complete varieties $X , \tX$ over $\oQ_p$. For any morphism $\alpha : \teX \to \eX$ over $\oZ_p$ the following diagram of categories and functors commutes (up to canonical isomorphisms):
\[
\xymatrix{
\Bh^s (\eX_{\eo}) \ar[r]^-{\rho_n} \ar[d]_{\alpha^*} & \Rep_{\Pi_1 (X)} (\eo_n) \ar[d]^{A (\alpha)} \\
\Bh^s (\teX_{\eo}) \ar[r]^-{\rho_n} & \Rep_{\Pi_1 (\tX)} (\eo_n) \; .
}
\]
\end{prop}

\begin{proof}
Commutativity on objects is equivalent to Proposition \ref{cd-t29} since $(\rho_n \verk \alpha^*)(\Eh)= \rho_{\alpha^* \Eh , n}$ and $(A (\alpha) \verk \rho_n) (\Eh) = \rho_{\Eh , n} \verk \alpha_*$. Commutativity on morphisms follows because for $f : \Eh_1 \to \Eh_2$ the morphism $(\rho_n \verk \alpha^*) (f) = \rho_{\alpha^* f, n}$ equals $(A (\alpha) \verk \rho_n) (f) = A (\alpha) \rho_{f,n}$ by our definitions. 
\end{proof}

We now pass to the projective limit of the $\mod p^n$ parallel transports. A good trivializing cover for $\Eh_{n+1}$ over a proper pair $(\eX , U)$ is also a trivializing cover for $\Eh_n$ over $(\eX , U)$. It follows that for an \'etale path $\gamma$ from $x \in X (\C_p)$ to $x' \in X (\C_p)$ and a bundle $\Eh$ in $\Bh^s (\eX_{\eo})$ the reduction modulo $ p^n$ of
\[
\rho_{\Eh , n+1} (\gamma) : \Eh_{x_{n+1}} \longrightarrow \Eh_{x'_{n+1}}
\]
equals
\[
\rho_{\Eh , n} (\gamma) : \Eh_{x_n} \longrightarrow \Eh_{x'_n} \; .
\]
Hence we may form the projective limit
\[
\rho_{\Eh} (\gamma) = \varprojlim_n \rho_{\Eh_n} (\gamma) : \Eh_{x_{\eo}} \longrightarrow \Eh_{x'_{\eo}} \; .
\]
Here $x_{\eo} \in \eX_{\eo} (\eo) = X (\C_p)$ is the extension  of $x$ to a section $x_{\eo} : \spec \eo \to \eX_{\eo}$ and $\Eh_{x_{\eo}} = x^*_{\eo} \Eh = \varprojlim_n \Eh_{x_n}$ is the corresponding free $\eo$-module. 

In this way we obtain a continuous functor
\[
\rho_{\Eh} : \Pi_1 (X) \longrightarrow \Free_r (\eo) \quad \text{where} \; r = \rank \, \Eh \; .
\]
With notation as in Propositions \ref{cd-t29} and \ref{cd-t30} we have $\rho_{\alpha^* \Eh} = \rho_{\Eh, n} \verk \alpha_*$ for any morphism $\alpha : \teX \to \eX$ and $\rho_{{}^{\sigma}\!\Eh} = \sigma_* \verk \rho_{\Eh} \verk \sigma^{-1}_*$ for $\sigma \in G_{\Q_p}$. 

\begin{prop} \label{prop-rho} By the previous construction, we  get an $\eo$-linear, exact functor 
\[
\rho : \Bh^s (\eX_{\eo}) \longrightarrow \Rep_{\Pi_1 (X)} (\eo)
\]
into the category of continuous functors from $\Pi_1 (X)$ to $\Free (\eo)$. Here $\Free (\eo)$ is the topological category of free $\eo$-modules of finite rank. The functor $\rho$ commutes with tensor products, duals, internal homs and exterior powers. Moreover for any morphism $\alpha : \teX \to \eX$ over $\oZ_p$ of models $\eX$ and $\teX$ with smooth proper generic fibres $X$ and $\tX$ we have a canonical isomorphism of functors $\rho \verk \alpha^* = A (\alpha) \verk \rho$ with $A (\alpha)$ similarly defined as in Proposition \ref{cd-t31}. 
\end{prop}

\begin{proof}: This follows from the construction and Theorem \ref{cd-t30n}.
\end{proof}

Let $\emm$ be the maximal ideal of $\eo$. We end this section with an alternative description of the reduction modulo $ \emm$ of $\rho_{\Eh,x_{\eo}}$. 

For a vector bundle $\Eh$ in $\Bh^s (\eX_{\eo})$ and a point $x \in X (\C_p)$ we obtain from $\rho_\Eh$ a continuous representation
\[
\rho_{\Eh , x_{\eo}} : \pi_1 (X,x) \longrightarrow \GL (\Eh_{x_{\eo}}) := \Aut_{\eo} (\Eh_{x_{\eo}}) \; .
\]
The reduction of $\rho_{\Eh,x}$ modulo $ \emm$ is the representation
\[
\rho_{\Eh , x_{\eo}} \otimes k : \pi_1 (X,x) \longrightarrow \GL (\Eh_{x_k})
\]
obtained by reducing $\rho_{\Eh , x_{\eo}} (\gamma) $ modulo $ \emm$ for every $\gamma \in \pi_1 (X,x)$. Here $x_k$ is the image of $x_{\eo}$ under reduction $\eX (\eo) \to \eX (k)$ where $k = \eo / \emm \cong \oF_p$. Clearly
\[
\rho_{\Eh, x_{\eo}} \otimes k = \rho_{\Eh , n , x} \otimes k \quad \text{for any} \; n \ge 1 \; ,
\]
where the right hand side is also defined by reduction.

By \cite{SGA1}, Exp. IX, Proposition 1.7 the inclusion $\eX^{\red}_k \hookrightarrow \eX_k$ induces an isomorphism $\pi_1 (\eX^{\red}_k , x_k) \iso \pi_1 (\eX_k , x_k)$. Moreover the inclusion $\eX_k \hookrightarrow \eX$ induces a natural isomorphism $\pi_1 (\eX_k , x_k) \iso \pi_1 (\eX , x_k)$. This follows from \cite{SGA1} Exp. X, Th\'eor\`eme 2.1 together with an argument to reduce the finitely presented case to a Noetherian situation as in the proof of \cite{SGA1} exp. IX, Th\'eor\`eme 6.1, p. 254 above. We also have a canonical isomorphism
\[
\pi_1 (\eX , x_k) = \Aut (F_{x_k}) = \Aut (F_x) = \pi_1 (\eX , x) \; .
\]
The reason is this: For any finite \'etale covering $\Yh \to \eX$, by the infinitesimal lifting property any point $y_k \in \Yh_k (k)$ over $x_k$ determines a unique section $y_{\eo} \in \Yh (\eo)$ over $x_{\eo} \in \eX (\eo)$ and hence a point $y \in Y (\C_p)$ over $x \in X (\C_p)$. This gives a natural bijection between the points $y_k$ over $x_k$ and the points $y$ over $x$. Thus the two fiber functors $F_{x_k}$ and $F_x$ on the finite \'etale coverings of $\eX$ are isomorphic. In conclusion we get a natural isomorphism 
\[
\pi_1 (\eX^{\red}_k , x_k) \iso \pi_1 (\eX_k , x_k) = \pi_1 (\eX , x_k) = \pi_1 (\eX , x) 
\]
and hence a specialization homomorphism
\[
\ssp_k : \pi_1 (X,x) \longrightarrow \pi_1 (\eX , x) = \pi_1 (\eX^{\red}_k , x_k) \; .
\]
Note that example 5 in \cite{De} shows that in general the representation $\rho_{\Eh , x_{\eo}}$ does not factor over $\ssp_k$. However the representation $\rho_{\Eh, x_{\eo}} \otimes k$ does factor over $\ssp_k$ as specified in the next theorem. The restriction $\Eh^{\red}_k$ of $\Eh_k$ to $\eX^{\red}_k$ is numerically flat by definition. Hence $\Eh^{\red}_k$ determines an algebraic representation of the $S$-fundamental group $\pi^S_1 (\eX^{\red}_k , x_k)$ on $(\Eh^{\red})_{x_k} = \Eh_{x_k}$. By Corollary \ref{cd-t5} we have $\pi^S_1 (\eX^{\red}_k , x_k)\cong \pi^N_1 (\eX^{\red}_k , x_k)$ if we assume that $\eX^{\red}_k$ is projective. Hence we get in this case an algebraic representation
\[
\blambda : \pi^N_1 (\eX^{\red}_k , x_k) \longrightarrow \bGL_{\Eh_{x_k}} \; .
\]
The group of $k$-valued points of the Nori fundamental group is the Grothendieck fundamental group. Hence we get a continuous representation where $\GL (\Eh_k)$ carries the discrete topology
\[
\lambda = \blambda (k) : \pi_1 (\eX^{\red}_k , x_k) \longrightarrow \bGL_{\Eh_k} (k) = \GL (\Eh_k) \; .
\]

\begin{theorem} \label{cd-t32}
Let $X$ be a smooth, complete variety over $\oQ_p$, $\eX$ a model of $X$ over $\oZ_p$ with projective special fibre $\eX_k$, $\Eh$ a vector bundle in $\Bh^s (\eX_{\eo})$ and $x \in X (\C_p)$. With notation as above the following diagram commutes:
\[
\xymatrix{
\pi_1 (X,x) \ar[r]^{\rho_{\Eh , x_{\eo}} \otimes k} \ar[d]_{\ssp_k} & \GL (\Eh_{x_k}) \ar@{=}[d] \\
\pi_1 (\eX^{\red}_k , x_k) \ar[r]^{\lambda} & \GL (\Eh_{x_k}) \; .
}
\]
\end{theorem}

\begin{proof} The proof of the corresponding result for curves given in \cite{De}, Theorem 1 can be adapted to the higher dimensional situation using the considerations in section \ref{sec:cd5}. We omit the details. 
\end{proof}
\section{Parallel transport on $\bm{X_{\C_p}}$} \label{sec:cd4}
In this section we construct parallel transport for suitable vector bundles over $\C_p$. 

\begin{defn} Let $X$ be a smooth and complete variety over $\oQ_p$.

i) We denote by $\Bh^s (X_{\C_p})$  the (full) category of vector bundles $E$ on $X_{\C_p}$ which up to isomorphism have an extension to a vector bundle $\Eh$ in $\Bh^s (\eX_{\eo})$ for some model $\eX$ over $\oZ_p$ of $X$. 

ii) Let $E$ be a vector bundle in $\Bh^s (X_{\C_p})$, and let $j_{\eX} : X \hookrightarrow \eX$ be the open immersion given by i). We define a functor \[
j^*_{\eX_{\eo}} : \Bh^s (\eX_{\eo}) \otimes \Q \longrightarrow \Bh^s (X_{\C_p}) \; .
\]
by pullback via  $j_{\eX_{\eo}} = j_{\eX} \otimes \eo : X_{\C_p} \hookrightarrow \eX_{\eo}$.
Here for an additive category $\Bh$ the category $\Bh \otimes \Q$ has the same objects as $\Bh$, and the morphism groups are the ones of $\Bh$ tensored with $\Q$. 
\end{defn}

Note that by flat base change the functor $j^*_{\eX_{\eo}}$ is fully faithful making $\Bh^s (\eX_{\eo}) \otimes \Q$ a full subcategory of $\Bh^s (X_{\C_p})$. By definition every bundle $E$ in $\Bh^s (X_{\C_p})$ is isomorphic to a bundle of the form $j^*_{\eX_{\eo}} \Eh$ with $\Eh$ in $\Bh^s (\eX_{\eo})$ for some model $\eX$ of $X$ over $\oZ_p$. The direct limit of the subcategories $\Bh^s (\eX_{\eo}) \otimes \Q$ is filtered in the following sense.

\begin{theorem} \label{cd-t33}
Let $X$ be a smooth complete variety over $\oQ_p$ with models $\eX_1$ and $\eX_2$ over $\oZ_p$. Then there is another model $\eX_3$ of $X$ together with morphisms
\[
\eX_1 \xleftarrow{p_1} \eX_3 \xrightarrow{p_2} \eX_2
\]
restricting to the identity on the generic fibres (after their identifications with $X$). This gives rise to a commutative diagram of fully faithful functors
\[
\xymatrix{
\Bh^s (\eX_{1\eo}) \otimes \Q \ar[dr]^{p^*_1} \ar@/^1pc/[drr]^{j^*_{\eX_{1\eo}}} \\
& \Bh^s (\eX_{3\eo}) \otimes \Q \ar[r]^{j^*_{\eX_{3\eo}}} & \Bh^s (X_{\C_p}) \\
\Bh^s (\eX_{2\eo}) \otimes {\Q} \ar[ur]_{p^*_2} \ar@/_1pc/[urr]_{j^*_{\eX_{2\eo}}} 
}
\]
\end{theorem}

\begin{proof}
We only need to show the existence of $\eX_3$ dominating $\eX_1$ and $\eX_2$. By finite presentation we may assume (with an abuse of notation) that $X$ is a variety over a finite extension $K$ of $\Q_p$ in $\oQ_p$ and $\eX_1 , \eX_2$ are models of $X$ over $\eo_K$. Let $\eX_3$ be the closure of the image of the diagonal
\[
X \xrightarrow{\Delta} X \times_{\spec K} X \longrightarrow \eX_1 \times_{\spec \eo_K} \eX_2
\]
endowed with the reduced subscheme structure. Then $\eX_3$ is integral and proper and flat over $\spec \eo_K$. Moreover there are morphisms $\eX_1 \xleftarrow{p_1} \eX_3 \xrightarrow{p_2} \eX_2$ induced by the projections. Base changing from $\eo_K$ to $\eo$ we get the assertion. 
\end{proof}

\begin{cor} \label{cd-t34}
The additive $\C_p$-linear category $\Bh^s (X_{\C_p})$ is stable under extensions, tensor products, duals, internal homs and exterior powers.
\end{cor}

\begin{proof}
This follows from Propositions \ref{cd-t20} and \ref{cd-t33}. Also note that for vector bundles $\Eh_1 , \Eh_2$ on $\eX_{\eo}$ with restrictions to $E_1, E_2$ on $X_{\C_p}$ we have
\[
\Ext^1_{\eX_{\eo}} (\Eh_1 , \Eh_2) \otimes_{\eo} \C_p = \Ext^1_{X_{\C_p}} (E_1 , E_2) \; .
\]
Moreover, if $\Eh$ is a quasi-coherent sheaf on $\eX_{\eo}$ in an extension 
\[
0 \longrightarrow \Eh_2 \longrightarrow \Eh \longrightarrow \Eh_1 \longrightarrow 0 
\]
then $\Eh$ is a vector bundle since the extension splits locally.
\end{proof}

For every morphism $\alpha : \teX \to \eX$ of models of $X$, by Proposition \ref{prop-rho} we have a commutative diagram
\[
\xymatrix{
\Bh^s (\eX_{\eo}) \ar[rr]^{\alpha^*} \ar[dr]_{\rho} & & \Bh^s (\teX_{\eo}) \ar[dl]^{\rho} \\
 & \Rep_{\Pi_1 (X)} (\eo) \; . &
}
\]
Let $\Rep_{\Pi_1 (X)} (\C_p)$ be the category of continuous functors from $\Pi_1 (X)$ to the category $\nVec (\C_p)$ of finite dimensional $\C_p$-vector spaces. The parallel transport $\rho$ gives a $\C_p$-linear functor
\[
\rho \otimes \Q : \Bh^s (\eX_{\eo}) \otimes \Q \xrightarrow{\rho \otimes \Q} \Rep_{\Pi_1 (X)} (\eo) \otimes \Q \longrightarrow \Rep_{\Pi_1 (X)} (\C_p) \; .
\]
Now Proposition \ref{cd-t33} and the commutativity of the diagram
\[
\xymatrix{
\Bh^s (\eX_{\eo}) \otimes \Q \ar[rr]^{\alpha^*} \ar[dr]_{\rho \otimes \Q} & & \Bh^s (\teX_{\eo}) \otimes \Q \ar[dl]^{\rho \otimes \Q} \\
 & \Rep_{\Pi_1 (X)} (\C_p)
}
\]
imply the following result for any smooth proper variety $X$ over $\Q_p$. 

\begin{prop} \label{cd-t35}
There is a uniquely defined $\C_p$-linear continuous functor
\[
\rho : \Bh^s (X_{\C_p}) \longrightarrow \Rep_{\Pi_1 (X)} (\C_p)
\]
such that for every model $\eX$ over $\oZ_p$ of $X$ we have
\[
\rho \verk j^*_{\eX_{\eo}} = \rho \otimes \Q : \Bh^s (\eX_{\eo}) \otimes \Q \longrightarrow \Rep_{\Pi_1 (X)} (\C_p) \; .
\]
The functor $\rho$ is exact and respects tensor products, duals, internal homs and exterior powers.
\end{prop}

Let us write down the explicit description of the functor $\rho$ in Proposition \ref{cd-t35}. For a bundle $E$ in $\Bh^s (X_{\C_p})$ we obtain a continuous functor
\[
\rho_E : \Pi_1 (X) \longrightarrow \nVec (\C_p)
\]
by setting on the one hand
\[
\rho_E (x) = E_x \quad \text{for} \; x \in X (\C_p) = \Ob \Pi_1 (X) \; .
\]
On the other hand, for $x , x' \in X (\C_p)$ the continuous map
\[
\rho_E : \Mor_{\Pi_1 (X)} (x,x') \longrightarrow \Hom_{\C_p} (E_x , E_{x'}) 
\]
is given by
\[
\rho_E (\gamma) = \psi^{-1}_{x'} \verk (\rho_{\Eh} (\gamma) \otimes_{\Z} \Q) \verk \psi_x \; .
\]
Here we have chosen a model $\eX$ of $X$ over $\oZ_p$ and a bundle $\Eh$ in $\Bh^s (\eX_{\eo})$ together with an isomorphism $\psi : E \iso j^*_{\eX_{\eo}} \Eh$ of bundles over $X_{\C_p}$. Here $j_{\eX_{\eo}} : X_{\C_p} \to \eX_{\eo}$ is the base extension of the inclusion $j_{\eX} : X \to \eX$. Moreover $\psi_x$ is the map
\[
\psi_x = x^* (\psi) : E_x \longrightarrow (j^*_{\eX_{\eo}} \Eh)_x = \Eh_{x_{\eo}} \otimes_{\eo} \C_p = \Eh_{x_{\eo}} \otimes_{\Z} \Q \; .
\]
For a morphism $f : E \to E'$ in $\Bh^s (X_{\C_p})$ the morphism $\rho_f : \rho_E \to \rho_{E'}$ is given by the family of linear maps $f_x = x^* (f) : E_x \to E'_x$ for all $x \in X (\C_p)$. 

The following result concerns functoriality properties of our parallel transport.

\begin{theorem} \label{cd-t36}
i) Let $X , \tX$ be smooth, complete varieties over $\oQ_p$ and let $f : \tX \to X$ be a morphism over $\oQ_p$. Pullback by $f$ induces a functor $f^* : \Bh^s (X_{\C_p}) \to \Bh^s (\tX_{\C_p})$ and the following diagram commutes:
\[
\xymatrix{
\Bh^s (X_{\C_p}) \ar[r]^-{\rho} \ar[d]_{f^*} & \Rep_{\Pi_1 (X)} (\C_p) \ar[d]^{A (f)} \\
\Bh^s (\tX_{\C_p}) \ar[r]^-{\rho} & \Rep_{\Pi_1 (\tX)} (\C_p) \; .
}
\]
Here $A(f)$ is defined  in the same way as $A (\alpha)$ in Proposition \ref{cd-t31}.\\
ii) For every $\sigma$ in $G_{\Q_p}$ and each vector bundle $E$ in $\Bh^s (X_{\C_p})$ the conjugate bundle ${}^{\sigma}\!E$ is in $\Bh^s (\hsi X_{\C_p})$ and the following diagram commutes where $r = \rank\,E$:
\[
\xymatrix{
\Pi_1 (X) \ar[r]^{\rho_E} \ar[d]^{\wr}_{\sigma_*} & \nVec_r (\C_p) \ar[d]^{\wr \, \sigma_*} \\
\Pi_1 (\hsi X) \ar[r]^{\rho_{\hsi E}} & \nVec_r (\C_p) \; .
}
\]
Here $\nVec_r (\C_p)$ is the category of $r$-dimensional $\C_p$-vector spaces and $\sigma_* $ is given by $- \otimes_{\C_p , \sigma} \C_p$.
\end{theorem}

\begin{proof}
i) This follows from the $\eo$-version of Proposition \ref{cd-t31} using Lemma \ref{cd-t9} and Proposition \ref{cd-t33}.\\
ii) This follows from the $\eo$-version of Proposition \ref{cd-t30} again using Proposition \ref{cd-t33}.
\end{proof}

It is obvious that the functor $\rho : \Bh^s (X_{\C_p}) \to \Rep_{\Pi_1 (X)} (\C_p)$ is faithful. For a point $x \in X (\C_p)$ the functor
\[
\Rep_{\Pi_1 (X)} (\C_p) \longrightarrow \Rep_{\pi_1 (X,x)} (\C_p) \; , \; (V_y)_{y \in X (\C_p)} \longmapsto V_x
\]
into the category of continuous finite dimensional $\C_p$-representations of $\pi_1 (X,x)$ is an equivalence of categories as in \cite{May} Ch. 2, \S\,4. In particular, we obtain the following consequence:

\begin{cor} \label{cd-t37}
For a smooth, complete variety $X$ over $\oQ_p$ and any point $x \in X (\C_p)$ the fibre functor $\Bh^s (X_{\C_p}) \to \uVec_{\C_p} , E \mapsto E_x$ is faithful. The evaluation map
\[
\Gamma (X_{\C_p} , E) \longrightarrow E^{\pi_1 (X,x)}_x \; , \; s \longmapsto s (x)
\]
is injective.
\end{cor}

\begin{proof}
The first assertion is clear. For the second, note that the trivial line bundle $\Oh_{X_{\C_p}}$ is in $\Bh^s (X_{\C_p})$. Hence the map 
\[
\Gamma (X_{\C_p} , E) = \Hom_{X_{\C_p}} (\Oh , E) \longrightarrow \Hom_{\pi_1 (X,x)} (\C_p , E_x) = E^{\pi_1 (X,x)}_x
\]
is injective. 

Given a vector bundle $E$ in $\Bh^s (X_{\C_p})$ and $s_x \in E^{\pi_1 (X,x)}_x$ we can use the parallel transport $\rho_E$ to construct a set theoretical section $s : X (\C_p) \to E (\C_p)$ with $s (x) = s_x$ as follows. For $y \in X (\C_p)$ choose an \'etale path $\gamma_y$ from $x$ to $y$ and set $s (y) = \rho_E (\gamma_y) (s_x)$. Since $s_x$ is $\pi_1 (X,x)$-invariant the value of $s (y)$ is independent of the choice of $\gamma_y$. Applied to the bundle $E = \uHom (F,G)$ for bundles $F,G$ in $\Bh^s (X_{\C_p})$ it follows that to every equivariant homomorphism $\varphi_x \in \Hom_{\pi_1 (X,x)} (F_x ,G_x)$ we obtain a family of $\C_p$-linear maps $\varphi_y : F_y \to G_y$ for $y \in X (\C_p)$. It would be nice to have a direct argument that $s$ is algebraic resp. that the $\varphi_y$'s are the fibres of a morphism $\varphi : F \to G$ of vector bundles.  This seems to require $p$-adic Hodge theory as in  \cite{Fa2}  with the consequence that 
\[
\Gamma (X_{\C_p} , E) \iso E^{\pi_1 (X,x)}_x
\]
and more generally
\[
\Hom_{X_{\C_p}} (F,G) \iso \Hom_{\pi_1 (X,x)} (F_x , G_x)
\]
are isomorphisms. The latter assertion means that the functor
\[
\Bh^s (X_{\C_p}) \longrightarrow \Rep_{\pi_1 (X,x)} (\C_p)
\]
is fully faithful.
\end{proof}

We end this section with the following theorem which shows that numerically flat reduction implies numerical flatness.

\begin{theorem} \label{cd-t38}
The bundles $E$ in $\Bh^s (X_{\C_p})$ are numerically flat on $X_{\C_p}$.
\end{theorem}

\begin{rem}
If $X$ is a smooth projective variety over $\oQ_p$ the numerically flat bundles on $X$ are the ones with vanishing rational Chern classes which are semistable with respect to all (equivalently one) polarization, see \cite{S} Corollary 3.10 and Remark and \cite{DPS}, Theorem 1.18 and Corollary 1.19.
\end{rem}

\begin{proof}
By definition there are a model $\eX$ over $\oZ_p$ of $X$ and a vector bundle $\Eh$ in $\Bh^s (\eX_{\eo})$ with $\Eh \otimes_{\eo} \C_p = E$. Since $\Bh^s (\eX_{\eo}) = \Bh_p (\eX_{\eo})$ by Theorems \ref{cd-t22} and \ref{cd-t23} there exists a good trivializing cover $\pi : \Yh \to \eX$ for $\Eh_1 = \Eh \mbox{ mod }p$ over $(\eX,U)$ for some open smooth quasi-projective subscheme $U$ of $X$. Since $\pi_{\C_p} : Y_{\C_p} \to X_{\C_p}$ is proper and dominant, it  is surjective. Hence it suffices to show that $\pi^*_{\C_p} E$ is numerically flat on $Y_{\C_p}$. Replacing $X$ by $Y$ etc. we may therefore assume in addition that $\eX / \oZ_p$ is a good model and that $\Eh_1$ is a trivial bundle on $\eX_1 = \eX \otimes_{\oZ_p} \oZ_p / p = \eX_{\eo} \otimes_{\eo} \eo / p$. Let $\alpha : Z \to X_{\C_p}$ be a morphism from a smooth projective curve $Z$ over $\C_p$ to $X_{\C_p}$. We have to show that $\alpha^* E$ is semistable of degree zero. Let $K / \Q_p$ be a finite extension field in $\oQ_p$ such that $\eX$ is defined over $\eo_K$,  i.e. $\eX = \eX_{\eo_K} \otimes_{\eo_K} K$ for a proper flat $\eo_K$-scheme $\eX$ (with geometrically reduced special fibre). Since $\eo$ is the filtered union of finitely generated normal $\eo_K$-subalgebras $A$ we can spread out the situation over some such $A$  using finite presentation. To be precise, there exist:\\
i) a vector bundle $\Eh_A$ over $\eX_A = \eX_{\eo_K} \otimes A$ with $\Eh_A \otimes_A \eo = \Eh$ and such that $\Eh_A \otimes A_1$ is a trivial bundle on $\eX_{A_1} = \eX_A \otimes_A A_1$ where $A_1 = A/pA$, and \\
ii) an $A_K$-morphism $\alpha_{A_K} : Z_{A_K} \to X_{A_K} = \eX_A \otimes_A A_K = X_K \otimes_K A_K$ where $A_K = A \otimes_{\eo_K} K$ and $X_K = \eX_{\eo_K} \otimes_{\eo_K} K$ such that\\
a) $Z_{A_K} \otimes_{A_K} \C_p = Z$ and \\
b) $\alpha_{A_K} \otimes_{A_K} \C_p = \alpha : Z \to X_{\C_p}$.\\
Taking the closure $\Zh_A$ of $Z_{A_K}$ in $\eX_A$ with its reduced structure we obtain an $A$-morphism $\alpha_A : \Zh_A \to \eX_A$ from a proper $A$-scheme $\Zh_A$ such that $\alpha_A \otimes_A A_K = \alpha_{A_K}$. The vector bundle $\Fh_A = \alpha^*_A \Eh_A$ is trivial modulo $p$. 

Choose a prime ideal $\ep$ of height one in $A$ containing a prime element $\pi_K$ of $\eo_K$. In fact, any prime ideal $\ep$ in $A$ corresponding to the generic point of an irreducible component of $\spec A \otimes \eo_K / \pi_K \eo_K$ (which is non-empty since $A \subset \eo$) will do. Since $A$ is normal $A_{\ep} \subset \C_p$ is a discrete valuation ring. Moreover $\ep \supset \pi_K A \supset pA$. Let $R$ be the strict Henselization of $A_{\ep}$ in the algebraic closure of $Q (A) = \Quot (A)$ in $\C_p$. Then $R$ is a discrete valuation ring with quotient field $Q \subset \C_p$ and separably closed residue field $\kappa \supset \eo_K / \pi_K \eo_K$. Let $(\Zh , \Fh)$ be the base extension of $(\Zh_A , \Fh_A)$ via $A \subset R$. The restriction of $\Fh$ to the special fibre of $\Zh$ is trivial since $\Fh_A \otimes_A A_1$ is trivial and $A \subset R$ induces a map $A_1 \to R / p \to \kappa$. The generic fibre $\Zh_Q = \Zh \otimes_R Q$ is a smooth projective curve over $Q$ since $\Zh_Q \otimes_Q \C_p = Z_{A_K} \otimes_{A_K} \C_p = Z$ is a smooth projective curve. Possibly replacing $\Zh$ by the closure of $\Zh_Q$ in $\Zh$ we may assume that $\Zh$ is integral and flat over $R$. Then $\Zh$ is a model of $\Zh_Q$ over $R$. We can now apply \cite{dewe1} Proposition 14 and obtain that $\Fh_Q = \Fh \otimes_R Q$ is semistable of degree zero on $\Zh_Q$. It follows that $\Fh_Q \otimes_Q \C_p = \alpha^* E$ is semistable of degree zero on $Z$ as was to be shown.
\end{proof}

\begin{rem}
Theorem \ref{cd-t23} which requires some work could  easily have been avoided in the above proof since the simple descent argument in the beginning of section \ref{sec:cd5} shows that any bundle $\Eh$ in $\Bh^s (\eX_{\eo}) = \Bh_\emm (\eX_{\eo})$ lies in $\Bh_t (\eX_{\eo})$ for some $t \in \oZ_p$ with $|t|< 1$. Instead of working with $\Eh_1 = \Eh \mbox{ mod }p$ we would then work with $\Eh \mbox{ mod } t$. 
\end{rem}
\section{Line bundles and the category $\Bh^{\sharp} (X_{\C_p})$} \label{sec:cd8}

In this section we will enlarge the class of vector bundles for which parallel transport can be constructed. The resulting category $\Bh^{\sharp} (X_{\C_p})$ contains in particular all numerically flat line bundles on $X_{\C_p}$ i.e. all line bundles in $\Pic^{\tau} (X_{\C_p})$ if $X$ is projective. 
For $\Bh^s (X_{\C_p})$ itself this is not known in general.

\begin{defn} \label{cd:t8.1}
For a smooth complete variety $X$ over $\oQ_p$ let $\Bh^{\sharp} (X_{\C_p})$ be the full subcategory of vector bundles $E$ on $X_{\C_p}$ with the following property: For each closed point $x$ of $X$ there is a proper morphism $\alpha : X' \to X$ from a smooth complete variety $X'$ to $X$ such that:\\
i) $\alpha$ is finite \'etale surjective over an open neighborhood $U$ of $x$, and \\
ii) $\alpha^* E$ is in $\Bh^s (X'_{\C_p})$. \\
\end{defn}

As a special case of the definition note that a vector bundle $E$ is in $\Bh^{\sharp} (X_{\C_p})$ if $\alpha^* E$ is in $\Bh^s (X'_{\C_p})$ for some finite \'etale surjective morphism $\alpha : X' \to X$ from a smooth, complete variety $X'$ to $X$. 

By definition $\Bh^s (X_{\C_p})$ is a full subcategory of $\Bh^{\sharp} (X_{\C_p})$. The new category $\Bh^{\sharp}$ inherits all the good properties of $\Bh^s$ and contains all relevant line bundles:

\begin{theorem}
\label{cd:t8.2}
The additive $\C_p$-linear category $\Bh^{\sharp} (X_{\C_p})$ is stable under extensions, tensor products, duals, internal homs and exterior powers. Moreover it contains all numerically flat line bundles on $X_{\C_p}$. For a morphism $f : \tX \to X$ of smooth complete $\oQ_p$-varieties, pullback by $f$ induces a functor $f^* : \Bh^{\sharp} (X_{\C_p}) \to \Bh^{\sharp} (\tX_{\C_p})$. For $\sigma \in G_{\Q_p}$ and $E$ in $\Bh^{\sharp} (X_{\C_p})$ the conjugate bundle $^{\sigma}\!E$ is in $\Bh^{\sharp} (\; ^{\sigma}\!X_{\C_p})$.
\end{theorem}

\begin{proof}
Given finitely many bundles $E_1 , \ldots , E_n$ in $\Bh^{\sharp} (X_{\C_p})$ and a closed point $x$ of $X$ the morphism $\alpha : X' \to X$ in Definition \ref{cd:t8.1} can be chosen such that besides i) the bundles $\alpha^* E_1 , \ldots , \alpha^* E_n$ are in $\Bh^s (X'_{\C_p})$. This follows by a similar argument as in the proof of Lemma \ref{cd-t17}. Now, all assertions except the one about line bundles follow from Corollary \ref{cd-t34} and the functorialities of $\Bh^s$ in Theorem \ref{cd-t36}. For the assertion about line bundles, using Chow's lemma we may assume that $X$ is a smooth projective variety over $\oQ_p$. Let $A = \Alb_X = \widehat{\Pic^0_X}$ be the Albanese variety of $X$. It is an abelian variety over $\oQ_p$. The choice of a point $x \in X (\oQ_p)$ determines a morphism $f : X \to A$ with $f (x) = 0$ such that pullback along $f$ induces an isomorphism
\[
f^* : \Pic^0 (A_{\C_p}) \iso \Pic^0 (X_{\C_p}) \; .
\]
Descend $A$ to an abelian variety $A_K$ over a finite extension $K$ of $\Q_p$, choose a model $\Zh_{\eo_K}$ over $\eo_K$ of $A_K$ and set $\Zh = \Zh_{\eo_K} \otimes_{\eo_K} \eo_{\oK}$. By Lemma \ref{cd-t13} applied to the proper pair $(\Zh,A)$ there are a finite extension $L$ of $K$ and a very good normal model $\Yh_{\eo_L}$ over $\eo_L$ of an abelian variety $A_L$ over $L$ with $A = A_L \otimes_L \oQ_p$. By \cite{ray}, Corollaire (6.4.5), the group functor $\Pic^{\tau}_{\Yh_{\eo_L} / \eo_L}$ is representable by a separable group scheme. It is locally of finite presentation over $\eo_L$ by \cite{ray}, Th\'eor\`eme (1.5.1) and its generic fibre is $\Pic^0_{A_L/L}$. Hence the subgroup $\Pic^{\tau}_{\Yh_{\eo_L}/\eo_L} (\eo)$ is open in $\Pic^0_{A_L / L} (\C_p)$. By \cite{co}, Theorem 4.1, the cokernel is torsion. For any bundle $\tQ \in \Pic^0 (A_{\C_p})$ some power $\tQ^b$ therefore extends to a line bundle $\Mh$ on $\Yh = \Yh_{\eo_L} \otimes_{\eo_L} \eo_{\oK}$ with $\Mh_k \in \Pic^{\tau}_{\Yh_{\kappa} / \kappa} (k)$ where $\kappa$ is the residue field of $\eo_L$. Enlarging $b$  if necessary we may assume that $\Mh_k \in \Pic^0_{\Yh_{\kappa} / \kappa} (k) = \Pic^0 (\Yh_k)$. Hence the line bundle $\Mh_k$ is Nori-semistable on $\Yh_k$ and therefore $\Mh \in \Bh^s (\Yh_{\eo})$. It follows that for any $\tQ$ in $\Pic^0 (A_{\C_p})$ some power $\tQ^b$ is in $\Bh^s (A_{\C_p})$. With these preparations in place let $L$ be a numerically flat line bundle on $X_{\C_p}$. Then some power $L^a$ is in $\Pic^0 (X_{\C_p})$ and hence of the form $L^a = f^* (\tQ)$ for some $\tQ$ in $\Pic^0 (A_{\C_p})$. Choose $b$ as above such that $Q := \tQ^b$ is in $\Bh^s (A_{\C_p})$ and set $N = ab$. Then we have
\[
L^N = f^* (Q) \; .
\]
Define the smooth complete variety $X'$ by the cartesian diagram
\[
\xymatrix{
X' \ar[r]^{f'} \ar[d]_{\pi} & A \ar[d]^N \\
X \ar[r]^f & A \; .
}
\]
Then $\pi$ is finite \'etale, and since $Q \in \Pic^0 (A_{\C_p})$ we find
\[
\pi^* (L)^N = \pi^* f^* (Q) = f^{'*} N^* (Q) = f^{'*} (Q)^N \; .
\]
Hence we have
\[
\pi^* (L) = f^{'*} (Q) \otimes P 
\]
for some line bundle $P$ on $X'_{\C_p}$ with $P^N \cong \Oh$. By Kummer theory we have
\[
\Pic (X'_{\C_p})_N \iso H^1_{\et} (X'_{\C_p} , \mu_N) \cong H^1_{\et} (X' , \mu_N) \; .
\]
Let $X''$ be a connected component of the \'etale $\mu_N$-torsor over $X'$ corresponding to $P$. Then pullback by the resulting finite \'etale morphism $\pi' : X'' \to X'$ trivialises $P$. Setting $\pi'' = \pi \verk \pi'$ it therefore follows that
\[
\pi^{''*} (L) = \pi^{'*} f^{'*} (Q) \otimes \pi^{'*} P = \pi^{'*} f^{'*} (Q)  \; .
\]
Since $Q$ is in $\Bh^s (A_{\C_p})$ it follows from the functoriality of the categories $\Bh^s$ that $\pi^{''*} (L)$ is in $\Bh^s (X''_{\C_p})$. Since $\pi'' : X'' \to X$ is finite \'etale surjective it follows from the definition of $\Bh^{\sharp}$ that $L$ is in $\Bh^{\sharp} (X_{\C_p})$.

It remains to construct the parallel transport also  for the bundles $E$ in $\Bh^{\sharp} (X_{\C_p})$. Given $E$ there exist an open covering $\eU$ of $X$ and for each member $U$ of the covering $\eU$ a proper morphism $\alpha = \alpha_U : X' \to X$ from a smooth complete variety $X' = X'_U$ over $\oQ_p$ such that $\alpha^{-1} (U) \to U$ is finite \'etale surjective and $\alpha^* E$ is in $\Bh^s (X'_{\C_p})$. We may assume that $\eU$ is stable under finite intersections. Using the same arguments as in the proof of Lemma \ref{cd-t18} we may also assume that a finite group $G = G_U$ acts on $X'$ over $X$ such that $\alpha^{-1} (U) \to U$ is Galois with group $G$. The theory of the parallel transport for bundles in $\Bh^s (X'_{\C_p})$ gives representations attached to $\alpha^* E$:
\[
\rho_{\alpha^* E} : \Pi_1 (X') \longrightarrow \nVec_{\C_p} \; .
\]
There is a transitive system of isomorphism
\begin{equation}
\label{eq:cd16}
\rho_{\alpha^* E} \verk \sigma_* = \rho_{\sigma^* \alpha^* E} = \rho_{\alpha^* E}
\end{equation}
since $\alpha \verk \sigma = \alpha$. The inclusion $\alpha^{-1} (U) \subset X'$ induces a functor
\[
\Pi_1 (\alpha^{-1} (U)) \longrightarrow \Pi_1 (X') \; .
\]
Composition with $\rho_{\alpha^* E}$ induces functors
\[
\rho^U_{\alpha^* E} : \Pi_1 (\alpha^{-1} (U)) \longrightarrow  \nVec_{\C_p} \; ,
\]
for which the analogue of \eqref{eq:cd16} holds. Since the finite \'etale covering $\alpha^{-1} (U) \to U$ is Galois with group $G$, it follows from \cite{dewe1} Proposition 31 that there is a unique continuous functor
\[
\rho^U_E : \Pi_1 (U) \to \nVec_{\C_p}
\]
such that
\[
\rho^U_{\alpha^* E} = \rho^U_E \verk \alpha_* : \Pi_1 (\alpha^{-1} (U)) \longrightarrow \nVec_{\C_p} \; . 
\]
Using functoriality of parallel transport for $\Bh^s$ and the uniqueness assertion in \cite{dewe1} Proposition 31 one can show that the functors $\rho^U_E$ are compatible with inclusions $V \hookrightarrow U$ in $\eU$. Using Theorem \ref{cd-t24} it follows that they are all induced by a unique continuous functor
\begin{equation}
\label{eq:cd17}
\rho_E : \Pi_1 (X) \longrightarrow \nVec_{\C_p} \; .
\end{equation}
\end{proof}

\begin{theorem}
\label{cd:t8.3}
The parallel transport \eqref{eq:cd17} for bundles $E$ in $\Bh^{\sharp} (X_{\C_p})$ is independent of all choices in its construction. The resulting $\C_p$-linear functor
\[
\rho : \Bh^{\sharp} (X_{\C_p}) \longrightarrow \Rep_{\Pi_1 (X)} (\C_p)
\]
is exact and respects tensor products, duals, internal homs and exterior powers. Theorem \ref{cd-t36} and Corollary \ref{cd-t37} (with the same proof) continue to hold for $\Bh^{\sharp}$ instead of $\Bh^s$. 
\end{theorem}

The proof is straightforward.

For curves, one can extend parallel transport even to the category 
of vector bundles $E$ on $X_{\C_p}$ for which there exists a finite \emph{but not necessarily \'etale} morphism $\alpha : X' \to X$ from a smooth projective curve $X'$ over $\oQ_p$ such that $\alpha^* E$ is in $\Bh^s(X'_{\C_p})$, see \cite{dewe2}. The main ingredient of this proof is a "no-monodromy theorem" around the singular points of $\alpha$ on $X$.

 It seems conceivable that the theory of  parallel transport can similarly be extended to all bundles $E$ on $X_{\C_p}$ for smooth complete varieties $X$ for which there exists a surjective morphism $\alpha : X' \to X$ from a smooth complete variety $X'$ over $\oQ_p$ such that $\alpha^* E$ is in $\Bh^s (X'_{\C_p})$. 
This property would follow if parallel transport existed for all numerically flat bundles on $X_{\C_p}$ as one may hope. 
%\input{abelian-varieties}
%\input{lit}
%\bibliographystyle{alpha}
%\bibliography{refs}

\begin{thebibliography}{DMOS82}

\bibitem[AGT16]{agt}
Ahmed Abbes, Michel Gros, and Takeshi Tsuji.
\newblock {\em The {$p$}-adic {S}impson correspondence}, volume 193 of {\em
  Annals of Mathematics Studies}.
\newblock Princeton University Press, Princeton, NJ, 2016.

\bibitem[Bha15]{bh1}
Bhargav Bhatt.
\newblock {$p$}-divisibility for coherent cohomology.
\newblock {\em Forum Math. Sigma}, 3, 2015.

\bibitem[BLR90]{blr}
Siegfried Bosch, Werner L\"utkebohmert, and Michel Raynaud.
\newblock {\em N\'eron models}, volume~21 of {\em Ergebnisse der Mathematik und
  ihrer Grenzgebiete (3) [Results in Mathematics and Related Areas (3)]}.
\newblock Springer-Verlag, Berlin, 1990.

\bibitem[BM08]{BM}
Edward Bierstone and Pierre~D. Milman.
\newblock Functoriality in resolution of singularities.
\newblock {\em Publ. Res. Inst. Math. Sci.}, 44(2):609--639, 2008.

\bibitem[BS16]{bh2}
Bhargav Bhatt and Andrew Snowden.
\newblock Refined {A}lterations, 2016.
\newblock http://www-personal.umich.edu/~bhattb/math/alterationsepsilon.pdf.

\bibitem[Col89]{co}
Robert~F. Coleman.
\newblock Reciprocity laws on curves.
\newblock {\em Compositio Math.}, 72(2):205--235, 1989.

\bibitem[Den10]{De}
Christopher Deninger.
\newblock Representations attached to vector bundles on curves over finite and
  {$p$}-adic fields, a comparison.
\newblock {\em M\"unster J. Math.}, 3:29--41, 2010.

\bibitem[dJ96]{dej1}
A.~J. de~Jong.
\newblock Smoothness, semi-stability and alterations.
\newblock {\em Inst. Hautes \'Etudes Sci. Publ. Math.}, (83):51--93, 1996.

\bibitem[dJ97]{dej2}
A.~Johan de~Jong.
\newblock Families of curves and alterations.
\newblock {\em Ann. Inst. Fourier (Grenoble)}, 47(2):599--621, 1997.

\bibitem[DMOS82]{DM}
Pierre Deligne, James~S. Milne, Arthur Ogus, and Kuang-yen Shih.
\newblock {\em Hodge cycles, motives, and {S}himura varieties}, volume 900 of
  {\em Lecture Notes in Mathematics}.
\newblock Springer-Verlag, Berlin-New York, 1982.

\bibitem[Don85]{D}
S.~K. Donaldson.
\newblock Anti self-dual {Y}ang-{M}ills connections over complex algebraic
  surfaces and stable vector bundles.
\newblock {\em Proc. London Math. Soc. (3)}, 50(1):1--26, 1985.

\bibitem[DPS94]{DPS}
Jean-Pierre Demailly, Thomas Peternell, and Michael Schneider.
\newblock Compact complex manifolds with numerically effective tangent bundles.
\newblock {\em J. Algebraic Geom.}, 3(2):295--345, 1994.

\bibitem[DW05]{dewe1}
Christopher Deninger and Annette Werner.
\newblock Vector bundles on {$p$}-adic curves and parallel transport.
\newblock {\em Ann. Sci. \'Ecole Norm. Sup. (4)}, 38(4):553--597, 2005.

\bibitem[DW10]{dewe2}
Christopher Deninger and Annette Werner.
\newblock Vector bundles on {$p$}-adic curves and parallel transport {II}.
\newblock In {\em Algebraic and arithmetic structures of moduli spaces
  ({S}apporo 2007)}, volume~58 of {\em Adv. Stud. Pure Math.}, pages 1--26.
  Math. Soc. Japan, Tokyo, 2010.

\bibitem[Fal05]{Fa2}
Gerd Faltings.
\newblock A {$p$}-adic {S}impson correspondence.
\newblock {\em Adv. Math.}, 198(2):847--862, 2005.

\bibitem[Fal11]{fa3}
Gerd Faltings.
\newblock A {$p$}-adic {S}impson correspondence {II}: small representations.
\newblock {\em Pure Appl. Math. Q.}, 7(4, Special Issue: In memory of Eckart
  Viehweg):1241--1264, 2011.

\bibitem[FC90]{fach}
Gerd Faltings and Ching-Li Chai.
\newblock {\em Degeneration of abelian varieties}, volume~22 of {\em Ergebnisse
  der Mathematik und ihrer Grenzgebiete (3) [Results in Mathematics and Related
  Areas (3)]}.
\newblock Springer-Verlag, Berlin, 1990.
\newblock With an appendix by David Mumford.

\bibitem[Ful98]{F}
William Fulton.
\newblock {\em Intersection theory}, volume~2 of {\em Ergebnisse der Mathematik
  und ihrer Grenzgebiete. 3. Folge. A Series of Modern Surveys in Mathematics
  [Results in Mathematics and Related Areas. 3rd Series. A Series of Modern
  Surveys in Mathematics]}.
\newblock Springer-Verlag, Berlin, second edition, 1998.

\bibitem[Gir71]{Gir}
Jean Giraud.
\newblock {\em Cohomologie non ab\'elienne}.
\newblock Springer-Verlag, Berlin-New York, 1971.
\newblock Die Grundlehren der mathematischen Wissenschaften, Band 179.

\bibitem[Gro63]{EGA3}
A.~Grothendieck.
\newblock \'{E}l\'ements de g\'eom\'etrie alg\'ebrique. {III}. \'{E}tude
  cohomologique des faisceaux coh\'erents. {II}.
\newblock {\em Inst. Hautes \'Etudes Sci. Publ. Math.}, (17):91, 1963.

\bibitem[Gro66]{EGAIV}
A.~Grothendieck.
\newblock \'{E}l\'ements de g\'eom\'etrie alg\'ebrique. {IV}. \'{E}tude locale
  des sch\'emas et des morphismes de sch\'emas. {III}.
\newblock {\em Inst. Hautes \'Etudes Sci. Publ. Math.}, (28):255, 1966.

\bibitem[Kat73]{ka}
Nicholas~M. Katz.
\newblock {$p$}-adic properties of modular schemes and modular forms.
\newblock pages 69--190. Lecture Notes in Mathematics, Vol. 350, 1973.

\bibitem[Kim92]{Ki}
Shun-ichi Kimura.
\newblock Fractional intersection and bivariant theory.
\newblock {\em Comm. Algebra}, 20(1):285--302, 1992.

\bibitem[Knu83]{kn}
Finn~F. Knudsen.
\newblock The projectivity of the moduli space of stable curves. {II}. {T}he
  stacks {$M_{g,n}$}.
\newblock {\em Math. Scand.}, 52(2):161--199, 1983.

\bibitem[Lan04]{lan04}
Adrian Langer.
\newblock Semistable sheaves in positive characteristic.
\newblock {\em Ann. of Math. (2)}, 159(1):251--276, 2004.

\bibitem[Lan11]{L1}
Adrian Langer.
\newblock On the {S}-fundamental group scheme.
\newblock {\em Ann. Inst. Fourier (Grenoble)}, 61(5):2077--2119 (2012), 2011.

\bibitem[Lan12]{L2}
Adrian Langer.
\newblock On the {S}-fundamental group scheme. {II}.
\newblock {\em J. Inst. Math. Jussieu}, 11(4):835--854, 2012.

\bibitem[Liu02]{Liu}
Qing Liu.
\newblock {\em Algebraic geometry and arithmetic curves}, volume~6 of {\em
  Oxford Graduate Texts in Mathematics}.
\newblock Oxford University Press, Oxford, 2002.
\newblock Translated from the French by Reinie Ern\'e, Oxford Science
  Publications.

\bibitem[Liu06]{Liu2}
Qing Liu.
\newblock Stable reduction of finite covers of curves.
\newblock {\em Compos. Math.}, 142(1):101--118, 2006.

\bibitem[LL99]{liulo}
Qing Liu and Dino Lorenzini.
\newblock Models of curves and finite covers.
\newblock {\em Compositio Math.}, 118(1):61--102, 1999.

\bibitem[LS77]{LS}
Herbert Lange and Ulrich Stuhler.
\newblock Vektorb\"undel auf {K}urven und {D}arstellungen der algebraischen
  {F}undamentalgruppe.
\newblock {\em Math. Z.}, 156(1):73--83, 1977.

\bibitem[LSZ]{lsz}
Guitang Lan, Mao Sheng, and Kang Zuo.
\newblock Semistable {H}iggs bundles, periodic {H}iggs bundles and
  representations of algebraic fundamental groups.
\newblock https://arxiv.org/abs/1311.6424.

\bibitem[L{\"u}t93]{Luet}
W.~L{\"u}tkebohmert.
\newblock On compactification of schemes.
\newblock {\em Manuscripta Math.}, 80(1):95--111, 1993.

\bibitem[LZ17]{liuzhu}
Ruochuan Liu and Xinwen Zhu.
\newblock Rigidity and a {R}iemann-{H}ilbert correspondence for {$p$}-adic
  local systems.
\newblock {\em Invent. Math.}, 207(1):291--343, 2017.

\bibitem[May99]{May}
J.~P. May.
\newblock {\em A concise course in algebraic topology}.
\newblock Chicago Lectures in Mathematics. University of Chicago Press,
  Chicago, IL, 1999.

\bibitem[MR84]{MR}
V.~B. Mehta and A.~Ramanathan.
\newblock Restriction of stable sheaves and representations of the fundamental
  group.
\newblock {\em Invent. Math.}, 77(1):163--172, 1984.

\bibitem[Nor82]{N}
Madhav~V. Nori.
\newblock The fundamental group-scheme.
\newblock {\em Proc. Indian Acad. Sci. Math. Sci.}, 91(2):73--122, 1982.

\bibitem[OV07]{ogus}
A.~Ogus and V.~Vologodsky.
\newblock Nonabelian {H}odge theory in characteristic {$p$}.
\newblock {\em Publ. Math. Inst. Hautes \'Etudes Sci.}, (106):1--138, 2007.

\bibitem[Ray70]{ray}
M.~Raynaud.
\newblock Sp\'ecialisation du foncteur de {P}icard.
\newblock {\em Inst. Hautes \'Etudes Sci. Publ. Math.}, (38):27--76, 1970.

\bibitem[Ray72]{R}
M.~Raynaud.
\newblock Flat modules in algebraic geometry.
\newblock {\em Compositio Math.}, 24:11--31, 1972.

\bibitem[Sai09]{Sa}
Takeshi Saito.
\newblock Wild ramification and the characteristic cycle of an {$l$}-adic
  sheaf.
\newblock {\em J. Inst. Math. Jussieu}, 8(4):769--829, 2009.

\bibitem[SGA03]{SGA1}
{\em Rev\^etements \'etales et groupe fondamental ({SGA} 1)}, volume~3 of {\em
  Documents Math\'ematiques (Paris) [Mathematical Documents (Paris)]}.
\newblock Soci\'et\'e Math\'ematique de France, Paris, 2003.
\newblock S\'eminaire de g\'eom\'etrie alg\'ebrique du Bois Marie 1960--61.
  [Algebraic Geometry Seminar of Bois Marie 1960-61], Directed by A.
  Grothendieck, With two papers by M. Raynaud, Updated and annotated reprint of
  the 1971 original [Lecture Notes in Math., 224, Springer, Berlin; MR0354651
  (50 \#7129)].

\bibitem[Sim92]{S}
Carlos~T. Simpson.
\newblock Higgs bundles and local systems.
\newblock {\em Inst. Hautes \'Etudes Sci. Publ. Math.}, (75):5--95, 1992.

\bibitem[SP]{stacks}
The {S}tacks {P}roject.
\newblock http://stacks.math.columbia.edu/download/book.pdf.

\bibitem[Sub07]{Su}
S.~Subramanian.
\newblock Strongly semistable bundles on a curve over a finite field.
\newblock {\em Arch. Math. (Basel)}, 89(1):68--72, 2007.

\bibitem[UY86]{UY}
K.~Uhlenbeck and S.-T. Yau.
\newblock On the existence of {H}ermitian-{Y}ang-{M}ills connections in stable
  vector bundles.
\newblock {\em Comm. Pure Appl. Math.}, 39(S, suppl.):S257--S293, 1986.
\newblock Frontiers of the mathematical sciences: 1985 (New York, 1985).

\bibitem[Xu17]{xu}
Daxin Xu.
\newblock Transport parall\`ele et correspondance de {S}impson {$p$}-adique.
\newblock {\em to appear in Forum Math. Sigma}, 5, 2017.
\newblock http://dx.doi.org/10.1017/fms.2017.7.

\end{thebibliography}

%\input{address}
\end{document}